\documentclass[10pt]{article}
\usepackage{amssymb}

\usepackage{amsthm}
\usepackage{amsmath}
\usepackage{graphicx}
\usepackage{color}
\allowdisplaybreaks[4]

 \newtheorem{thm}{Theorem}[section]
 
 \newtheorem{lem}[thm]{Lemma}
 
 \theoremstyle{definition}

 \numberwithin{equation}{section}

\def\ve{{\varepsilon}}

\def\rmd{\mathrm{d}}

\newtheorem{theorem}{Theorem}[section]

\theoremstyle{definition}

\theoremstyle{remark}
\newtheorem{remark}[theorem]{Remark}

\begin{document}
\title{Global Well-posedness for the Three Dimensional Simplified Inertial Ericksen-Leslie Systems Near Equilibrium}

\author{Yuan Cai
\footnote{Department of Mathematics, The Hong Kong University of Science and Technology, Clear Water Bay, Kowloon, Hong Kong.
\textit{Email}: maycai@ust.hk}
\and
Wei Wang\footnote{School of Mathematical Sciences, Zhejiang University, Hangzhou 310027, P. R. China.
\textit{Email}: wangw07@zju.edu.cn}
}
\date{}
\maketitle

\begin{abstract}
We study  a simplified inertial Ericksen-Leslie system for the nematic liquid crystal flow, which can be viewed as a system coupling Navier-Stokes
equations and wave map equations. We prove the global existence of classical solution with initial data near equilibrium.
\end{abstract}

\maketitle





\section{Introduction}
The Ericksen-Leslie system is a hydrodynamical theory for nematic liquid crystals which was
established by Ericksen \cite{Er} and Leslie \cite{Le} in 1960's. It has been successful to model various dynamical behavior for nematic liquid crystals.
In this paper, we consider the simplified Ericksen-Leslie system \cite{Le} with the inertial term:
\begin{equation}\label{inertial-EL}
\begin{cases}
\partial_t v +v\cdot \nabla v+\nabla p=
\mu \Delta v-\nabla \cdot (\nabla d \otimes \nabla d),\\
\sigma_0 D_{t}^2 d+\sigma_1 D_td-\Delta d=(|\nabla d|^2-\sigma_0 |D_t d|^2)d,\\
\nabla \cdot v=0.
\end{cases}
\end{equation}
on $\mathbb{R}^n \times \mathbb{R}^+$ $(n\geq 2)$.
{Here $v(x,t)\in\mathbb{R}^n$ is the average velocity of liquid crystal molecules, $p$ is the pressure, and
$d(x,t)\in \mathbb{S}^2$ is the directional field representing the orientation alignment of liquid crystal molecules.}
 $D_t=(\partial_t +v\cdot \nabla )$ denotes the material derivative.
{ The first two equations in $\eqref{inertial-EL}$
 represent the conservation of momentum and angular momentum respectively. The third equation is the incompressibility condition for the velocity.
The term $(|\nabla d|^2-\sigma_0 |D_t d|^2)d$ on the right hand side of $\eqref{inertial-EL}_2$
plays the role of Lagrangian multiplier for the constraint $|d|=1$.}
The term $\sigma_0 D_{t}^2 d$ is called the inertial term, while $\sigma_1 D_td$ is called the damping term. If $\sigma_0=0,\sigma_1>0$,
then $\eqref{inertial-EL}_2$ is called the simplified (non-inertial) Ericksen-Leslie system which is
a parabolic type equation and  has been widely studied in literatures since the work of Lin etc. \cite{Lin89, LL95, LL96, LL00}. If $\sigma_0>0,\sigma_1=0$, then $\eqref{inertial-EL}_2$ becomes a hyperbolic one.
We remark that the system considered here neglects complicated Leslie's stress terms
in the momentum equation and corresponding co-rotational and stretching terms in the angular momentum equation, however,
the inertial term is kept. We refer to
{\cite{Er91, Le} or \cite{JL} for its mathematical derivation and the full form}.

There have been many works on the non-inertial Ericksen-Leslie system.
For the simplified system without Leslie's stress,
Lin-Lin-Wang \cite{LLW} and Hong \cite{Ho} established the existence of global weak solutions in $\mathbb{R}^2$, see also
Lin-Wang \cite{LinW}, Xu-Zhang \cite{XZ}, Hong-Xin \cite{HX} and Lei-Li-Zhang \cite{LLZ} for related results in the two dimensional case.
Recently, Lin-Wang \cite{LWcpam} proved the global existence of weak solution for dimension three when the initial alignments $d_0$ locates
in the upper half sphere. We also refer to \cite{LiW, HW} for global existence of strong solution with small data for dimension three.
For the full non-inertial system with Leslie's stress, Wang-Zhang-Zhang \cite{WZZ} proved the local existence for general data and global existence  of smooth solutions
for data near equilibrium under optimal constraints on the Leslie coefficients for three dimensional case.
Wang-Wang\cite{WW} extended these results to the general Oseen-Frank energy case.
For 2D case, the existence of global weak solutions was proved by Wang-Wang  \cite{WW} and Huang-Lin-Wang \cite{HLW}, while
the uniqueness of weak solutions was considered by Wang-Wang-Zhang \cite{WWZ} and Li-Titi-Xin \cite{LTX}.

To the best knowledge of the authors, the results considering the Ericksen-Leslie system with inertial term are very few, especially in space dimensions higher than one.
There are some results on the one dimensional problem with general Oseen-Frank energy, in which, however, the fluid coupling is neglected (i. e. $v=0$).
We refer to \cite{BZ, ZZ05, ZZ10, ZZ12} and references therein for examples.
For the full inertial Ericksen-Leslie system, Jiang-Luo \cite{JL}  studied the local wellposedness of classical solutions  as well as the global wellposedness when the damping coefficients $\sigma_1>0$.

In this work, we study the global wellposedness of classical solutions to the simplified inertial Ericksen-Leslie system \eqref{inertial-EL} in three dimensions for the case $\sigma_1=0$.
{We remark that the absence of the term $\sigma_1 D_td$ makes the system more difficult to study than the case of $\sigma_1>0$ which bring an additional damping effect.}
In general, for quasi-linear wave equations, the solution may blow up in finite time even if the initial data are smooth and small.
However, due to the null structure nature of the nonlinear term in $\eqref{inertial-EL}_2$, one can expect existences of global smooth solutions with small initial data.
{Actually, if the fluid velocity is ignored, the system is just a equation for wave maps, which
the null condition is satisfied and is now well known to have global small solutions.}

In addition, we assume $\sigma_0=1,\ \mu=1$ without loss of generality.
The main result of this paper is stated as follows (the notations will be explained in Section 2.2):
\begin{thm}\label{thm}
Let
$(v_0(x),d_0(x)-e)\in$
{$H^\kappa_\Lambda$} with $\kappa\geq 9$ where $e\in \mathbb{S}^2$ is some constant director.
Suppose
$$\|(v_0,d_0-e,d_1)\|^2_{H^\kappa_\Lambda}\leq \epsilon.
$$
There exists a positive constant
$\epsilon_0$ such that, if $\epsilon \leq \epsilon_0$, the system (\ref{inertial-EL}) with initial data
$$v(x,0)=v_0(x),\quad d(x,0)=d_0(x),\quad \partial_t d(x,0)=d_1(x)$$
has a unique global classical solution such that
\begin{equation}\label{thmpri}
E^d_{\kappa+1}(t)\leq C_0 \epsilon \langle t\rangle^{\delta},\quad E^v_{\kappa}+E^d_{\kappa-1}(t)\leq C_0\epsilon
\end{equation}
for some positive $C_0>0$, $0<\delta< \frac12$ depending on $\kappa$ and $\epsilon_0$ uniformly for all $0\le t<\infty$.
\end{thm}

\begin{remark}
Following the recent works \cite{CLLM, Lei16, LW}, we can apply the ghost weight energy method \cite{Alinhac01a} to even prove a uniform bound for the highest-order energy.
However,  as the current estimate is sufficient to obtain the global existences of solutions, we will not pursue this direction in this paper.
\end{remark}
\begin{remark}
The regularity index $\kappa$ may be lowered a little bit. However, restricted by vector field method, we can not expect it close to the natural energy space.
Hence, we do not aim to find the lowest regularity index here.
\end{remark}
\begin{remark}
Inspired by the threshold global regularity on wave maps \cite{SterbenzTataru1, SterbenzTataru2}, it might be natural to conjecture that the similar results also hold for the two dimensional inertial liquid crystal. This issue will be addressed in our future work.
\end{remark}

We will perform the analysis under Eulerian coordinates, although Lagrangian coordinates may also work.
The main difficulty in proving global solutions lies in the contradiction between parabolicity and the hyperbolicity. In another words, the parabolic method and the hyperbolic approach are not compatible.
To overcome this essential difficulty, we need use an unified approach working both in parabolic equations and hyperbolic ones. This is the main novelty of this paper.

Considering the wave nature of the equation for $d$, we use the vector field theory and the weighted energy method. However, the
system has neither Lorentz invariance nor scaling invariance. We use the weighted $L^2$ norm introduced by Klainerman and Sideris \cite{KlainermanSiderisS96}
to overcome the difficulty from the lack of Lorentz invariance, while the lack of scaling invariance is {solved} by directly applying the scaling operator and dealing with the commutators \cite{Kessenich}.
To close the energy estimate, it is necessary to obtain the subcritical decay for the solution. However, the interaction between velocity filed and the orientation field weakens the dissipative nature of the velocity field and the dispersive effect of the orientation filed, and even worse, it is strengthened by the quasilinear nature of the system. These difficulties make the decay estimate very delicate.
In addition,  to close the higher-order energy estimate, we need to explore the symmetry structure of the system to deal with the quasilinear terms which may cause loss of derivatives at a first glance. This structure is not
obvious in the Eulerian formulation.

Before ending the introduction, we simply review some related results concerning the wave maps
 which corresponds to the inertial Ericksen-Leslie system \eqref{inertial-EL} when velocity is equal zero and
 $\sigma_0=1$, $\sigma_1=0$.
The Cauchy problem for the wave maps has been extensively studied, especially when space dimension is two, where the scaling is invariant in the natural energy space.
In the supercritical case $n\geq 3$, it was observed by Shatah \cite{Shatah88} that there
exist self-similar blow up solutions of finite energy.
In the critical case n = 2, Struwe \cite{Struwe} observed
that in the equivariant setting, blow up in this dimension must result from { a strictly slower
than self-similar rescaling of a harmonic sphere of finite energy}.
Finally, the blow up for this more difficult two-dimensional case was by
 exhibited by Rodnianski-Sterbenz \cite{RodnianskiSterbenz}, as well as Krieger-Schlag-Tataru \cite{KST}.
As for the well-posedness results, the wave map displays a null form structure, which was the essential feature in the subcritical theory of Klainerman-Machedon \cite{KlainermanMachedon1, KlainermanMachedon2, KlainermanMachedon3}, and Klainerman-Selberg \cite{KlainermanSelberg97,KlainermanSelberg02}.
These authors proved strong local well-posedness for data in $H^s(\mathbb{R}^n)$ when $s > n/2$.
The important critical theory $s = n/2$ was begun
by Tataru \cite{TataruLocal, TataruBesov}.
 In a breakthrough work, Tao \cite{TaoI, TaoII}
 proved the global regularity for small initial data in
 $\dot{H}^{\frac n2}(\mathbb{R}^n) \times \dot{H}^{\frac n2-1}(\mathbb{R}^n)$ when the target is a sphere by renormalization of the geometric structure.
Afterwards-Klainerman-Rodnianski \cite{KlaRod},
Nahmod-Stephanov-Uhlenbeck \cite{NSU}, Tataru \cite{TataruRough,Tatara04}, and
Krieger \cite{Krieger03I, Krieger03II, Krieger04} considered
other cases of targets by using similar methods as in Tao's work.
As for large initial data, Sterbenz-Tataru \cite{SterbenzTataru1, SterbenzTataru2} gave the following satisfactory result: If the energy of the initial data is smaller than the energy of any
nontrivial harmonic map $\mathbb{R}^n\rightarrow M$ , then one has the global existence theory.
A particular case is the hyperbolic space $\mathbb{H}^n$ for
which Tao \cite{ TaoIII, TaoIV, TaoV, TaoVI, TaoVII} as well as Krieger-Schlag \cite{KS} has achieved the same result.

The remaining part of this paper is organized as follows. In the next section,
we will introduce the vector fields applied onto the system, ansatz for the method of continuity, and some preliminary estimates.
In Section 3, we will estimate of weighted $L^2$ norm.
Section 4 is devoted to the decay estimate of velocity. In the last
section, we give various energy estimates which are crucial steps to prove Theorem 1.

\section{Preliminaries}
\subsection{Application of vector fields}

As stated in the introduction, we will use various vector fields which play a central role in
the proofs.
To begin with, we rewrite the system \eqref{inertial-EL} as
\begin{equation}\label{LiquidHyper}
\begin{cases}
\partial_t v +v\cdot \nabla v+\nabla p=
\mu \Delta v-\nabla \cdot (\nabla d \otimes \nabla d),\\[-4mm]\\
\partial_{tt} d-\Delta d=(|\nabla d|^2-|\partial_t d|^2-|v\cdot\nabla d|^2-2 v\cdot\nabla d\cdot\partial_t d)d\\[-4mm]\\
\qquad\qquad\quad-(2v\cdot\nabla\partial_t d+\partial_tv\cdot\nabla d+v\cdot\nabla(v\cdot\nabla d)),\\[-4mm]\\
\nabla \cdot v=0.
\end{cases}
\end{equation}
We have seven nonlinear terms in the equations for the orientation field, which make the whole argument a little annoying.

Now let us take a look at the invariant groups of the system \eqref{LiquidHyper}.
Suppose that $(u(t, x), d(t,x))$ is a solution of \eqref{LiquidHyper},
then one can check that $(Q^\top u(t, Qx), d(t,Qx))$ is also a solution of \eqref{LiquidHyper}
for any orthogonal matrix $Q$. We choose $Q$
to be one parameter group generated by anti-symmetric matrices:
\begin{equation*}
 Q = e^{\lambda
A_{i}},\quad \forall\ 1\leq i\leq 3,
\end{equation*}
where
\begin{equation*}
\begin{split}
&A_{1}=e_2\otimes e_3- e_3\otimes e_2,\\
&A_{2}=e_3\otimes e_1- e_1\otimes e_3,\\
&A_{3}=e_1\otimes e_2- e_2\otimes e_1.
\end{split}
\end{equation*}
The perturbed angular momentum
operators are defined  as infinitesimal generators of the orthogonal groups:
\begin{equation*}
\begin{cases}
\widetilde{\Omega}_{i} v=\Omega_{i} v+A_{i}v,\\
\widetilde{\Omega}_{i} d=\Omega_{i} d,
\end{cases}
\end{equation*}
where $\Omega=(\Omega_1,\Omega_2,\Omega_3)$ is the rotational gradient operator defined by
\begin{equation*}
\Omega = x \wedge \nabla.
\end{equation*}
Schematically, we have the commutation:
\begin{equation}\label{commuR}
[\partial,\Omega]=\partial,
\end{equation}
where $\partial$  means the span of $\{\partial_t, \partial_1, \partial_2, \partial_3\}$.
\eqref{commuR} and the following commutation \eqref{commuS} will be often used implicitly.

Putting $(Q^\top u(t, Qx), d(t,Qx))$ into $\eqref{LiquidHyper}$, differentiating them with respect to $\lambda$ and taking $\lambda \rightarrow 0$, we have
\begin{equation}\label{LiquidHyper_Rot}
\begin{cases}
\partial_t \widetilde{\Omega} v +\widetilde{\Omega}v\cdot \nabla v+v\cdot \nabla \widetilde{\Omega}v+\nabla \widetilde{\Omega}p\\[-4mm]\\
\qquad\qquad=\mu \Delta \widetilde{\Omega}v-  \nabla \cdot (\nabla \widetilde{\Omega} d \otimes \nabla d)-  \nabla \cdot (\nabla d \otimes \nabla \widetilde{\Omega} d),\\[-4mm]\\
\partial_{tt} \widetilde{\Omega} d-\Delta \widetilde{\Omega} d=
2(\nabla \widetilde{\Omega} d\cdot\nabla d-\partial_t \widetilde{\Omega} d\cdot \partial_t d)d\\[-4mm]\\
\qquad\qquad\quad -(2 \widetilde{\Omega}v\cdot\nabla d\cdot\partial_t d
+2 v\cdot\nabla \widetilde{\Omega}d\cdot\partial_t d
+2 v\cdot\nabla d\cdot\partial_t \widetilde{\Omega}d)d \\[-4mm]\\
\qquad\qquad\quad-2\big[ (\widetilde{\Omega} v\cdot \nabla) d \cdot(v\cdot \nabla) d
+ (v\cdot \nabla)\widetilde{\Omega} d \cdot(v\cdot \nabla) d \big]d \\[-4mm]\\
\qquad\qquad\quad+
(|\nabla d|^2-|\partial_t d|^2-|v\cdot\nabla d|^2-2 v\cdot\nabla d\cdot\partial_t d)\widetilde{\Omega}d \\[-4mm]\\
\qquad\qquad\quad-(2\widetilde{\Omega}v\cdot\nabla\partial_t d+2v\cdot\nabla\partial_t \widetilde{\Omega}d
+\partial_t\widetilde{\Omega}v\cdot\nabla d +\partial_tv\cdot\nabla \widetilde{\Omega}d )\\[-4mm]\\
\qquad\qquad\quad-\big[\widetilde{\Omega}v\cdot\nabla(v\cdot\nabla d)
+v\cdot\nabla(\widetilde{\Omega}v\cdot\nabla d)
+v\cdot\nabla(v\cdot\nabla \widetilde{\Omega}d)\big].\\[-4mm]\\
\nabla \cdot \widetilde{\Omega}v=0.
\end{cases}
\end{equation}

Next, we try to apply the scaling operator
\begin{equation*}
S = t\partial_t +r\partial_r
\end{equation*}
onto the system.
Unfortunately, the system \eqref{LiquidHyper} does not have any scaling invariance (nor Lorentz invariance).
Inspired by \cite{CLLM, Kessenich}, however, we are still able to use the scaling operator under this circumstance.

Applying $S+1$ onto $\eqref{LiquidHyper}_1$, $\eqref{LiquidHyper}_3$, and applying $S+2$ onto ${\eqref{LiquidHyper}}_2$, thanks to the commutation:
\begin{equation} \label{commuS}
(S+1)\partial=\partial S, \quad (S+2)\partial^2=\partial^2 S,
\end{equation}
 we can derive by direct calculations  that
\begin{equation}\label{LiquidHyper_Scale}
\begin{cases}
\partial_t S v +Sv\cdot \nabla v+v\cdot \nabla Sv+\nabla Sp \\[-4mm]\\
\qquad\qquad=
\mu \Delta (S-1) v-  \nabla \cdot (\nabla (S-1) d \otimes \nabla d)-  \nabla \cdot (\nabla d \otimes \nabla (S-1) d),\\[-4mm]\\
\partial_{tt} S d-\Delta S d=
2(\nabla S d\cdot\nabla d-\partial_t S d\cdot \partial_t d)d\\[-4mm]\\
\qquad\qquad\quad -(2 Sv\cdot\nabla d\cdot\partial_t d
+2 v\cdot\nabla Sd\cdot\partial_t d
+2 v\cdot\nabla d\cdot\partial_t Sd)d \\[-4mm]\\
\qquad\qquad\quad-2\big[ (S v\cdot \nabla) d \cdot(v\cdot \nabla) d
+ (v\cdot \nabla)S d \cdot(v\cdot \nabla) d \big]d \\[-4mm]\\
\qquad\qquad\quad+
(|\nabla d|^2-|\partial_t d|^2-|v\cdot\nabla d|^2-2 v\cdot\nabla d\cdot\partial_t d)Sd \\[-4mm]\\
\qquad\qquad\quad-(2Sv\cdot\nabla\partial_t d+2v\cdot\nabla\partial_t Sd
+\partial_tSv\cdot\nabla d +\partial_tv\cdot\nabla Sd )\\[-4mm]\\
\qquad\qquad\quad-\big[Sv\cdot\nabla(v\cdot\nabla d)
+v\cdot\nabla(Sv\cdot\nabla d)
+v\cdot\nabla(v\cdot\nabla Sd)\big].\\[-4mm]\\
\nabla \cdot Sv=0.
\end{cases}
\end{equation}
Note that there are some commutators in the equation for velocity in
\eqref{LiquidHyper_Scale} due to the appearance of viscosity.
This is different from the application of rotation operator.

Now we are going to apply compositions of generalized operators.
Let 
\begin{equation*}
\Gamma \in \{\partial_t,\partial_1,\partial_2,\partial_3,
\widetilde{\Omega}_1,\widetilde{\Omega}_2,\widetilde{\Omega}_3\},
\end{equation*}
and $Z^a=S^{a_1}\Gamma^{a'} $, where $a=(a_1,a_2,...,a_8)=(a_1,a')$,
$\Gamma^{a'}=\Gamma^{a_2}\Gamma^{a_3}...\Gamma^{a_8}$.
Using the reduction argument, we can use \eqref{LiquidHyper_Rot} and \eqref{LiquidHyper_Scale} to derive that
\begin{equation}\label{LiquidHyper_GeDe}
\begin{cases}
\partial_t Z^a v
-\mu \Delta (S-1)^{a_1}\Gamma^{a'} v =f^1_a, \\[-4mm]\\
\partial_{tt} Z^a d-\Delta Z^a d= f^2_a, \\[-4mm]\\
\nabla \cdot Z^a v=0.
\end{cases}
\end{equation}
where
\begin{equation}\label{f1f2}
\begin{cases}
f^1_a=-\sum_{b+c=a} C_a^b  Z^b v\cdot \nabla Z^c v
-\nabla Z^a p \\[-4mm]\\
\qquad-\sum_{b+c=a} C_a^b
\nabla \cdot (\nabla (S-1)^{b_1}\Gamma^{b'} d \otimes \nabla(S-1)^{c_1}\Gamma^{c'} d),\\[-4mm]\\
f^2_a=\sum_{b+c+e=a} C_a^{b,c}
(\nabla Z^b d\cdot\nabla Z^c d-\partial_t Z^b d\cdot \partial_t Z^c d) Z^e d \\[-4mm]\\
\qquad-\sum_{b+c+e+f=a} C_a^{b,c,e} \big[2 (Z^b v\cdot\nabla) Z^c d\cdot\partial_t Z^e d \big] Z^f d\\[-4mm]\\
\qquad-\sum_{b+c+e+f+g=a}C_a^{b,c,e,f}
\big[  ( Z^b v\cdot \nabla) Z^c d \cdot ((Z^e v\cdot \nabla) Z^f d)  \big] Z^g d \\[-4mm]\\
\qquad-\sum_{b+c+e=a}C_a^{b,c}Z^b v\cdot\nabla( Z^c v\cdot \nabla Z^e d) \\[-4mm]\\
\qquad-\sum_{b+c=a} C_a^b
(2Z^b v\cdot\nabla\partial_t Z^c d+\partial_t Z^b v\cdot\nabla Z^c d).
\end{cases}
\end{equation}
Here $C_a^b$, $C_a^{b,c}$, $C_a^{b,c,e}$ and $C_a^{b,c,e,f}$ are multinomial coefficients:
\begin{align*}
&C_a^b=\frac{a !}{b ! (a-b)!},
\quad C_a^{b,c}=\frac{a !}{b !c! (a-b-c)!},\\
& C_a^{b,c,e}=\frac{a !}{b !c! e!(a-b-c-e)!},
\quad C_a^{b,c,e,f}=\frac{a !}{b !c! e! f!(a-b-c-e-f)!}.
\end{align*}
The above commutation relations \eqref{LiquidHyper_GeDe}-\eqref{f1f2} are the starting points of this paper.

\subsection{Some notations}
Throughout this paper, we use the generalized energy defined by
\begin{equation*}
E^v_{\kappa}(t) = \|Z^\kappa v(t,\cdot)\|_{L^2}^2,
\quad
E_{\kappa+1}^d(t) = \|\partial Z^{\kappa} d(t,
\cdot)\|_{L^2}^2.
\end{equation*}
where $Z^\kappa v=\{Z^a v: |a|\leq \kappa \}$,
$Z^\kappa d=\{Z^a d: |a|\leq \kappa \}$.

We also use the weighted energy norm of Klainerman-Sideris \cite{KlainermanSiderisS96}:
\begin{equation*}\label{WEnergy}
X^d_\kappa(t) = \| \langle r-t\rangle \partial^2 Z^{\kappa-2} d\|^2_{L^2},
\end{equation*}
in which $\langle \sigma\rangle = \sqrt{1 + \sigma^2}$, for $\kappa\geq 2$.

To describe the space of  initial data, we introduce (see \cite{Sideris00})
\begin{equation}\nonumber
\Lambda = \{\nabla, \widetilde\Omega,r\partial_r\},
\end{equation}
and
\[
H^\kappa_\Lambda =\{(f,g,h):\sum_{|a|\le \kappa}\|\Lambda^a f\|_{L^2}+\|\nabla \Lambda^a g\|_{L^2}+\|\Lambda^a h\|_{L^2}<\infty\},
\]
with the norm
\[
\|(f,g,h)\|_{H^\kappa_\Lambda} =\sum_{|a|\le \kappa}\big(\|\Lambda^a f\|_{L^2}+\|\nabla \Lambda^a g\|_{L^2}+\|\Lambda^a h\|_{L^2} \big),
\]
for the vector of $f$, $g$ and $h$.

Throughout the whole paper, we will use $A\lesssim B$ to denote $A\leq C B$ for some positive absolute constant $C$,
whose meaning may change from line to line. We remark that, without specification, the constant  depends only on $\kappa$, but not on $t$.

\subsection{Ansatz for the method of continuity}
The local well-posedness has been proved recently in \cite{JL}. To extend the local solution to be a global one, we need to show the uniform estimate in time.

We make the following ansatz for the generalized energy:
\begin{eqnarray*}
E^v_{\kappa}\leq C\epsilon, \quad
E^d_{\kappa+1}\leq C\epsilon\langle t\rangle^{\delta},  \quad
E^d_{\kappa-1}\leq C\epsilon,
\end{eqnarray*}
for $\kappa\geq 9$, which is \eqref{thmpri} in Theorem \ref{thm}. Under the above a priori assumption,
we first show that the weighted $L^2$ norm $X_{\kappa}$ can be controlled by the generalized energy:
\begin{equation*}
X^d_{\kappa-1}\lesssim E^d_{\kappa-1}, \quad
X^d_{\kappa+1}\lesssim E^d_{\kappa+1},
\end{equation*}
see Section 3. Next, we show the decay estimate for the velocity in Section 4:
\begin{eqnarray}
&&\| \langle t\rangle^{\frac34}v\|_{L^\infty_x}\lesssim \epsilon^{\frac12}, \quad
\| \langle t\rangle^{\frac54}\nabla v\|_{L^\infty_x}\lesssim \epsilon^{\frac12}, \nonumber\\[-4mm]\nonumber\\
&&\|\langle t \rangle^{\frac32} \nabla^a v\|_{L^\infty_{x}} \leq \epsilon^{\frac12} (\ln {\langle t\rangle})^{\frac12},
\quad \forall\  2\leq |a| \leq  \kappa-3, \nonumber\\[-4mm]\nonumber\\
&&\|\partial_t v\|_{L^\infty_x}\lesssim\langle t\rangle^{-\frac32} (\ln \langle t \rangle)^{\frac12}\epsilon^{\frac12}. \label{prioridec}
\end{eqnarray}
Then, 
we can finally close the energy estimates by establishing the following inequalities:
\begin{align}
& E^v_{\kappa}(t)
+ \mu \|\nabla Z^k v(\cdot,\cdot) \|^2_{L^2_tL^2_x} \nonumber\\
&\quad\ \qquad\lesssim E^v_{\kappa}(0)+\int_0^t\langle \tau \rangle^{-2} E^d_{\kappa+1}(\tau) E^d_{\kappa-1}(\tau) \ \rmd \tau+\|\nabla Z^k v(\cdot,\cdot) \|^2_{L^2_tL^2_x}
E^v_{\kappa}(t). \label{priorivel} \\
&\frac{d}{dt}E^d_{\kappa+1} \lesssim \langle t\rangle^{-1}
(E^v_{\kappa}+E^d_{\kappa-1})^{\frac{1}{2}}
E^d_{\kappa+1}+\|\nabla Z^\kappa v\|^2_{L^2} E^d_{\kappa+1}\nonumber\\[-4mm]\nonumber\\
&\quad\qquad \qquad+\langle t\rangle^{-1}
\|\nabla Z^\kappa v\|_{L^2} E^d_{\kappa+1}
+\|\partial_t v\|_{L^\infty}E^d_{\kappa+1}. \label{prioridh} \\
&\frac{d}{dt}E^d_{\kappa-1}
\lesssim \langle t\rangle^{-1} \|\nabla Z^{\kappa} v\|^{\frac{1}{2}}_{L^2}(E^v_{\kappa})^{\frac14}(E^d_{\kappa-1}E^d_{\kappa+1})^{\frac12}\nonumber\\
&\qquad\qquad\quad+\langle t\rangle^{-\frac32} E^d_{\kappa-1} (E^d_{\kappa+1})^{\frac12}
+\| \nabla Z^{\kappa} v\|^2_{L^2}E^d_{\kappa-1}. \label{prioridl}
\end{align}
This is the main topic in Section 5.

With the help of above estimates  \eqref{prioridec}-\eqref{prioridl},  we can prove the main theorem of the paper.
\begin{proof}[Proof of Theorem \ref{thm}]
It suffices to show that under the bootstrap assumption:
\begin{equation}\label{pri1}
E^d_{\kappa+1}(t)\leq C_0 \epsilon \langle t\rangle^{\delta},\quad E^v_{\kappa}(t)+ \frac12\mu \|\nabla Z^k v(\cdot,\cdot) \|^2_{L^2_tL^2_x}+E^d_{\kappa-1}(t)\leq C_0\epsilon\ll 1,
\end{equation}
for some positive $C_0>0$, $0<\delta<\frac12$ depending on $\kappa$ and $\epsilon_0$ uniformly for $0\le t<T$, we can derive a stronger estimate:
\begin{equation}\label{pri2}
E^d_{\kappa+1}(t)\leq \frac12 C_0 \epsilon \langle t\rangle^{\delta},
\quad E^v_{\kappa}(t)+ \frac12\mu \|\nabla Z^k v(\cdot,\cdot) \|^2_{L^2_tL^2_x}+E^d_{\kappa-1}(t)\leq \frac12 C_0\epsilon.
\end{equation}
Then we can use continuity argument to extend life span of the solutions.

Firstly, under the assumption of \eqref{pri1}, we have from \eqref{priorivel} that
\begin{align*}
E^v_{\kappa}(t)+ \frac12\mu \|\nabla Z^k v(\cdot,\cdot) \|^2_{L^2_tL^2_x}
&\leq CE^v_{\kappa}(0)+C\int_0^t\langle \tau \rangle^{-2} E^d_{\kappa+1}(\tau) E^d_{\kappa-1}(\tau) \ \rmd \tau\nonumber\\
&\leq CE^v_{\kappa}(0)+C\int_0^t\langle \tau \rangle^{\delta-2} C^2_0\epsilon^2   \ \rmd \tau \nonumber\\
&\leq C\epsilon+ CC^2_0\epsilon^2.
\end{align*}
Let $C_0$ and $\epsilon$ such that
\begin{equation}\label{condition1}
\mathrm{max}\{8, 8C\}\leq  C_0,\quad  8CC_0 \epsilon \leq 1,
\end{equation}
then $E^v_{\kappa}(t)+ \frac12\mu \|\nabla Z^k v(\cdot,\cdot) \|^2_{L^2_tL^2_x}\leq \frac14 C_0\epsilon$.

Secondly, under the assumption of \eqref{pri1}, we can derive from \eqref{prioridh} that
\begin{eqnarray*}
\frac{d}{dt}E^d_{\kappa+1} \leq C\big(\langle t\rangle^{-1}
(E^v_{\kappa}+E^d_{\kappa-1})^{\frac{1}{2}}
+\|\nabla Z^\kappa v\|^2_{L^2} +\langle t\rangle^{-1}\|\nabla Z^\kappa v\|_{L^2}
+\|\partial_t v\|_{L^\infty} \big)E^d_{\kappa+1}. \label{prioridh1}
\end{eqnarray*}
Then, Gronwall inequality gives us that
\begin{align*}
E^d_{\kappa+1}(t)&\leq  E^d_{\kappa+1}(0)\exp\Big(C\int_0^t
\langle \tau \rangle^{-1}
(E^v_{\kappa}+E^d_{\kappa-1})^{\frac{1}{2}}(\tau)
+\|\nabla Z^\kappa v(\tau)\|^2_{L^2} \\
&\qquad\qquad+\langle \tau\rangle^{-1}\|\nabla Z^\kappa v(\tau)\|_{L^2}
+\|\partial_t v\|_{L^\infty}  d\tau\Big)\\
&\leq \epsilon\exp \Big( C(C_0\epsilon)^{\frac12}\ln\langle t\rangle +CC_0\epsilon+C(C_0\epsilon)^{\frac12}  +C\epsilon^{\frac12}\Big)\\
&\leq \epsilon \exp(C (C_0\epsilon)^\frac12 ) \langle t\rangle^{CC_0^{\frac12}\epsilon^{\frac12}}.
\end{align*}
Taking $\epsilon$ small enough such that
\begin{equation}\label{condition2}
\exp(C (C_0\epsilon)^\frac12 )\leq \frac12C_0, \quad CC_0^{\frac12}\epsilon^{\frac12}\leq \delta\leq \frac12,
\end{equation}
we have $E^d_{\kappa+1}(t)\leq \frac12 C_0\epsilon \langle t \rangle^{\delta}$.

Finally, using Holder's inequality, we derive from \eqref{prioridl} that
\begin{eqnarray*}
\frac{d}{dt}E^d_{\kappa-1}
&\leq& C\big(\langle t\rangle^{-1} \|\nabla Z^{\kappa}v\|^{\frac{1}{2}}_{L^2}(E^d_{\kappa+1})^{\frac12}
+\langle t\rangle^{-\frac32}  (E^d_{\kappa+1})^{\frac12}
+\| \nabla Z^{\kappa} v\|^2_{L^2}\big)E^d_{\kappa-1}\nonumber\\
&&+C\langle t\rangle^{-1} \|\nabla Z^{\kappa}v\|^{\frac{1}{2}}_{L^2}(E^v_{\kappa})^{\frac12}(E^d_{\kappa+1})^{\frac12}. \label{prioridl1}
\end{eqnarray*}
Then the Gronwall inequality gives us that
\begin{align*}
E^d_{\kappa-1}(t)\leq
&\Big(E^d_{\kappa-1}(0)+C\int_0^t \langle \tau \rangle^{-1} \|\nabla Z^{\kappa}v(\tau)\|^{\frac{1}{2}}_{L^2}(E^v_{\kappa}(\tau))^{\frac12}(E^d_{\kappa+1}(\tau))^{\frac12} d\tau\Big) \nonumber\\
&\cdot\exp\Big(C\int_0^t \langle \tau\rangle^{-1} \|\nabla Z^{\kappa}v(\tau)\|^{\frac{1}{2}}_{L^2}(E^d_{\kappa+1}(\tau))^{\frac12}
+\langle \tau\rangle^{-\frac32}  (E^d_{\kappa+1}(\tau))^{\frac12}
+\| \nabla Z^{\kappa} v(\tau)\|^2_{L^2} d\tau\Big) \nonumber\\
&\leq (\epsilon+C (C_0\epsilon)^{\frac32}  )\exp ( CC_0\epsilon+CC_0^{\frac12}\epsilon^{\frac12} ).
\end{align*}
As $c_0>8$, we can choose $\epsilon$ small enough such that
\begin{equation}\label{condition3}
(\epsilon+C (C_0\epsilon)^{\frac32}  )\exp( CC_0\epsilon+CC_0^{\frac12}\epsilon^{\frac12} )\leq \frac14C_0\epsilon,
\end{equation}
then we have $E^d_{\kappa-1}(t) \leq\frac14C_0\epsilon$.

Therefore, if we choose appropriate $C_0$ and small $\epsilon$ such that \eqref{condition1}, \eqref{condition2} and
\eqref{condition3} holds, then a better estimates \eqref{pri2} can be obtained.
Thus the theorem is proved.
\end{proof}

\subsection{Preliminary Weighted Estimates}

In this section, we list a few weighted estimates, which will be frequently used through out this paper.

First, we give two weighted $L^\infty-L^2$ estimates of the unknown near the light cone. They are essentially due to Klainerman and Sideris \cite{KlainermanSiderisS96}.
\begin{lem}\label{lemmaK-S-3D-1}
Let $u\in H^2(\mathbb{R}^3)$, then there hold
\begin{align}
\langle r\rangle^{1/2}|u(x)| & \lesssim\sum_{|\alpha|\leq 1}\|\nabla \widetilde{\Omega}^\alpha u\|_{L^2},
\label{K-S-3D-1}\\
\langle r\rangle |u(x)| & \lesssim\sum_{|\alpha|\leq 1} \|\partial_r \widetilde{\Omega}^\alpha u\|_{L^2}^{1/2}\cdot
\sum_{|\alpha|\leq 2} \|\widetilde{\Omega}^\alpha u\|_{L^2}^{1/2}, \label{K-S-3D-2}
\end{align}
provided the right hand side is finite.
\end{lem}
\begin{proof}
For \eqref{K-S-3D-1}, see Lemma 4.2 in \cite{KlainermanSiderisS96}. For \eqref{K-S-3D-2}, see Lemma 3.3 in \cite{Sideris00}.
\end{proof}

Next, we present some weighted $L^\infty-L^2$ estimate away from the light cone.

\begin{lem}\label{lemmaLW}
Let $u\in H^2(\mathbb{R}^3) $, then
there hold
\begin{eqnarray}
&&\langle t \rangle \| u(t, \cdot)\|_{L^\infty(r\leq\langle t\rangle/2)}
 \lesssim \|u\|_{L^2}+\| \langle t-r\rangle \nabla u\|_{L^2}+\| \langle t-r\rangle\nabla^2 u\|_{L^2},\label{KSWeiIn1}\\[-4mm]\nonumber\\
 &&\langle t \rangle \| u(t, \cdot)\|_{L^6(r\leq\langle t\rangle/2)}
 \lesssim \|u\|_{L^2}+\| \langle t-r\rangle \nabla u\|_{L^2}, \label{KSWeiIn2} \\[-4mm]\nonumber\\
&&\langle t\rangle^{\frac12}\| u(t,\cdot)\|_{L^3(r\leq \langle t\rangle/2)}\lesssim
\| u \|^{\frac12}_{L^2(\mathbb{R}^3)}
(\|\langle r-t\rangle \nabla  u\|_{L^2(\mathbb{R}^3)}
+\| u\|_{L^2(\mathbb{R}^3)}  )^{\frac12},\label{KSWeiIn3}
\end{eqnarray}
provided the right hand side is finite.
\end{lem}
The first inequality \eqref{KSWeiIn1} comes from  \cite{LW} of Lemma 4.3.
\begin{remark}
This lemma depends on the fact that the spatial dimension is three or higher. In the two dimensional case, the conclusion would {be weaker.}
\end{remark}

\begin{proof}
These three inequalities follow from the following Sobolev imbedding respectively: $\| u \|_{L^\infty(\mathbb{R}^3)}\lesssim \| \nabla u\|^\frac12_{L^2(\mathbb{R}^3)}\| \nabla^2 u\|^\frac12_{L^2(\mathbb{R}^3)}$,
$\| u \|_{L^6(\mathbb{R}^3)}\lesssim \| \nabla u\|_{L^2(\mathbb{R}^3)}$, and $\|u\|_{L^3(\mathbb{R}^3) }\lesssim \|u\|^{\frac12}_{L^2(\mathbb{R}^3)} \|\nabla u\|^{\frac12}_{L^2(\mathbb{R}^3) }$, providing $u\in{L^2(\mathbb{R}^3)}$. Due to the similarity of these inequalities, we only present a detailed proof for the third one.

Choose a radial cut-off function $\phi\in C^\infty(\mathbb{R}^3)$ which satisfies
\begin{equation}\nonumber
\phi(x) =
\begin{cases}
1,\quad {\rm if}\ r \leq \frac{1}{2}\\
0, \quad {\rm if}\ r \geq \frac{2}{3}
\end{cases},
\quad |\nabla\phi| \lesssim 1.
\end{equation}
For each fixed $t \geq 1$, let $\phi^t(x) = \phi(x/\langle
t\rangle)$. Clearly, one has
$$\phi^t(x) \equiv 1\ \ {\rm for}\ r \leq
\frac{\langle t \rangle}{2},\quad \phi^t(x) \equiv 0\ \ {\rm
for}\ r \geq\frac{2\langle t \rangle}{3}$$ and
$$|\nabla\phi^t(x)| \lesssim
\langle t\rangle^{-1}.$$
Consequently,
\begin{align*}
\| u\|_{L^3(r\leq \langle t\rangle/2)}
&\leq \|\phi^t u\|_{L^3(\mathbb{R}^3)}  \\[-4mm]\\
&\lesssim
\|\phi^t u \|^{\frac12}_{L^2(\mathbb{R}^3)}
(\|\phi^t\nabla u\|_{L^2(\mathbb{R}^3)}
+\langle t\rangle^{-1}\| u\|_{L^2(\mathbb{R}^3)}  )^{\frac12}\\
&\lesssim \langle t\rangle^{-\frac12}
\| u \|^{\frac12}_{L^2(\mathbb{R}^3)}
(\|\langle r-t\rangle \nabla  u\|_{L^2(\mathbb{R}^3)}
+\| u\|_{L^2(\mathbb{R}^3)}  )^{\frac12},
\end{align*}
which yields \eqref{KSWeiIn3}.
\end{proof}

Now we state two lemmas of weighted estimates using the structure of wave type equations, which can be found in   \cite{KlainermanSiderisS96} and
\cite{Lei16}.
We only state them without giving details of the proof.
\begin{lem}\label{WE-1}
There holds
\begin{equation}\nonumber
X^d_{2} \lesssim E^{d}_2 +
\|\langle t+r\rangle(\partial_t^2 - \Delta)d\|_{L^2},
\end{equation}
provided the right hand side is finite.
\end{lem}
\begin{proof}
See Lemma 2.3 and Lemma 3.1 in \cite{KlainermanSiderisS96}.
\end{proof}

Near the light cone, the good unknown $(\partial_t+\partial_r)d$ has better decay. The following lemma comes from \cite{Lei16} where the two dimension case was proved.
Indeed, it holds for all dimension $n\geq 2$.
\begin{lem}\label{GoodDeri}
For $\frac{\langle t\rangle}{2} \leq r$, there holds
\begin{equation}\nonumber
\langle t\rangle| (\partial_t + \partial_r) \partial d| \lesssim
|\nabla d| + |\nabla Z d| + t|(\partial_t^2 - \Delta) d)|.
\end{equation}
\end{lem}
\begin{proof}
See Lemma 3.4 in \cite{Lei16}. The proof given in \cite{Lei16} is for
$\frac{\langle t\rangle}{2} \leq r\leq \frac{5\langle t\rangle}{2}$ with space dimension $n=2$. However, one can easily check that the proof is valid for $r\geq \frac{\langle t\rangle}{2}$ and $n\geq 2$.
\end{proof}

\section{Estimates of the weighted $L^2$ norm}
This section is devoted to the estimate of weighted $L^2$ norm $X_{\kappa}$. To this end,
we need to estimate the  $L^2$ norm of $f^2_a$ with some weights.
\begin{lem} \label{lemWeNo}
For all multi-index $a$,  there holds
\begin{align*}
&\| (t+r)f_{a}^2\|^2_{L^2} \nonumber\\
&\lesssim E^d_{|a|+1}(X^d_{[|a|/2]+3}+E^d_{[|a|/2]+3})(E^d_{[|a|/2]+3}+1) \nonumber\\
&\quad+ E^v_{|a|+1} (E^v_{[|a|/2]+3}+1) (E^d_{[|a|/2]+4}+X^d_{[|a|/2]+4}) \nonumber\\
&\quad + (E^d_{|a|+2}+X^d_{|a|+2})E^v_{[|a|/2]+3}(E^v_{[|a|/2]+3}+1) \nonumber\\
&\quad+(E^v_{|a|}E^d_{[|a|/2]+3}+E^v_{[|a|/2]+2}E^d_{|a|+1})  (E^v_{[|a|/2]+3}+1)(E^d_{[|a|/2]+3}+X^d_{[|a|/2]+3}) (E^d_{[|a|/2]+3}+1),
\end{align*}
provided the right hand side is finite.
\end{lem}
\begin{proof}
Recalling the definition of $f_a^2$ in \eqref{f1f2}, we write
\begin{align*}
&\sum_{b+c+e=a} C_a^{b,c} \|(t+r)
(\nabla Z^b d\cdot\nabla Z^c d-\partial_t Z^b d\cdot \partial_t Z^c d) Z^e d \|^2_{L^2} \\
&+\sum_{b+c=a} C_a^b \|(t+r)(2Z^b v\cdot\nabla\partial_t Z^c d)\|^2_{L^2}\\
&+\sum_{b+c=a} C_a^b \|(t+r)(\partial_t Z^b v\cdot\nabla Z^c d)\|^2_{L^2}\\
&+\sum_{b+c+e=a}C_a^{b,c}\|(t+r)Z^b v\cdot\nabla( Z^c v\cdot \nabla Z^e d)\|^2_{L^2}\\
&+\sum_{b+c+e+f+g=a}C_a^{b,c,e,f}
\|(t+r)  ( Z^b v\cdot \nabla) Z^c d \cdot ((Z^e v\cdot \nabla) Z^f d)  \big] Z^g d\|^2_{L^2}\\
&+\sum_{b+c+e+f=a} C_a^{b,c,e} \|(t+r)\big[2 (Z^b v\cdot\nabla) Z^c d\cdot\partial_t Z^e d \big] Z^f d\|^2_{L^2}\\
&=H_1+H_2+H_3+H_4+H_5+H_6.
\end{align*}
In the sequel, we will focus our mind on $H_1$, $H_2$ and $H_3$, since these terms contain quadratic terms.
The remain terms $H_4, H_5, H_6$ are all cubic or higher order ones whose estimates are similar and easier.

We first estimate $H_1$:
\begin{align}\label{WeNo1}
 H_1\lesssim
&\big \| (t+r) |\partial Z^{|a|}d| |\partial Z^{[|a|/2]}d| |Z^{[|a|/2]}d| \big\|^2_{L^2} \nonumber\\[-4mm]\nonumber\\
&+\sum_{[a/2]\leq |b| \leq |a|}\| (t+r) |\partial Z^{[|a|/2]}d|^2 |Z^b d| \big\|^2_{L^2}.
\end{align}
If $r\geq \langle t\rangle/2$, by \eqref{K-S-3D-2}, the right hand side of \eqref{WeNo1} can be controlled by
\begin{align*}
& \|  \partial Z^{|a|}d\|^2_{L^2} \|r\partial Z^{[|a|/2]}d\|^2_{L^\infty(r\geq\langle t\rangle/2)}
  \|Z^{[|a|/2]}d \|^2_{L^\infty} \\[-4mm]\\
&+\sum_{[a/2]\leq |b| \leq |a|}
\| r \partial Z^{[a/2]}d\|^2_{L^\infty(r\geq\langle t\rangle/2)}
\| \partial Z^{[a/2]}d\|^2_{L^3(r\geq\langle t\rangle/2)}
\|Z^b d\|^2_{L^6} \\
&\lesssim E^d_{|a|+1}E^d_{[|a|/2]+3}( 1+ E^d_{[|a|/2]+3} ).
\end{align*}
Otherwise, if $r\leq \langle t\rangle/2$, by \eqref{KSWeiIn1} and \eqref{KSWeiIn2}, the right hand side of \eqref{WeNo1} can be controlled by
\begin{align*}
& \|  \partial Z^{|a|}d\|^2_{L^2} \|\langle t\rangle\partial Z^{[|a|/2]}d\|_{L^\infty(r\leq\langle t\rangle/2)}
  \|Z^{[|a|/2]}d \|^2_{L^\infty} \\[-4mm]\\
&+\sum_{[a/2]\leq |b| \leq |a|} \| \partial Z^{[a/2]}d\|^2_{L^6}
\| \langle t\rangle \partial Z^{[a/2]}d\|^2_{L^6 (r\leq\langle t\rangle/2)}
 \|Z^b d\|^2_{L^6} \\
&\lesssim E^d_{|a|+1}(X^d_{[|a|/2]+3}+E^d_{[|a|/2]+2})(E^d_{[|a|/2]+2}+1) .
\end{align*}
Thus we conclude that
\begin{align*}
H_1\lesssim E^d_{|a|+1}(X^d_{[|a|/2]+3}+E^d_{[|a|/2]+3})(E^d_{[|a|/2]+3}+1).
\end{align*}

For $H_2$, we have:
\begin{equation}\label{WeNo2}
H_2=4\sum_{b+c=a} C_a^b\| (t+r)
Z^b v\cdot\nabla\partial_t Z^c d \|^2_{L^2} .
\end{equation}
When $r\geq \langle t\rangle/2$, by \eqref{K-S-3D-2}, we can estimate the right hand side of \eqref{WeNo2} as
\begin{align*}
&\| r Z^{[|a|/2]}v \|^2_{L^\infty(r\geq \langle t\rangle/2)} \|\nabla\partial_t Z^{|a|} d\|^2_{L^2}
+\| Z^{|a|}v \|^2_{L^2}\|r \nabla\partial_t Z^{[|a|/2]} d\|^2_{L^\infty(r\geq \langle t\rangle/2)}\\[-4mm]\\
&\lesssim E^v_{[|a|/2]+2} E^d_{|a|+2} + E^v_{|a|} E^d_{[|a|/2]+4}.
\end{align*}
When $r\leq \langle t\rangle/2$, using \eqref{KSWeiIn1}, the right hand side of \eqref{WeNo2} can be controlled by
\begin{align*}
&\|  Z^{[|a|/2]}v \|^2_{L^\infty} \|\langle t-r \rangle \nabla\partial_t Z^{|a|} d\|^2_{L^2}
+\|  Z^{|a|}v \|^2_{L^2}\|\langle t\rangle \nabla\partial_t Z^{|a|} d\|^2_{L^\infty(r\leq \langle t\rangle/2)}\\[-4mm]\\
&\lesssim E^v_{[|a|/2]+2} X^d_{|a|+2} + E^v_{|a|} X^d_{[|a|/2]+4}+E^v_{|a|} E^d_{[|a|/2]+2}.
\end{align*}
Hence we conclude that
\begin{align*}
H_2\lesssim E^v_{[|a|/2]+2} E^d_{|a|+2} + E^v_{|a|} E^d_{[|a|/2]+4}
+E^v_{[|a|/2]+2} X^d_{|a|+2} + E^v_{|a|} X^d_{[|a|/2]+4}.
\end{align*}

Then we consider
\begin{align}\label{WeNo3}
&H_3\lesssim \big\|(t+r) |\partial_t  Z^{|a|} v|   |\nabla Z^{[|a|/2]} d| \big\|^2_{L^2}
+\big \| |(t+r)\partial_t Z^{[|a|/2]}v| |\nabla Z^{|a|}d|\big\|^2_{L^2}.
\end{align}
When $r\geq \langle t\rangle/2$, similar to the estimate of \eqref{WeNo2}, the right hand side of \eqref{WeNo3} can be bounded by
\begin{equation*}
 E^v_{|a|+1}E^d_{[|a|/2]+3}
+E^v_{[|a|/2]+3} E^d_{|a|+1}.
\end{equation*}
For the case of $r\leq \langle t\rangle/2$, with the help of \eqref{KSWeiIn1} and \eqref{KSWeiIn2},
we can estimate the right hand side of \eqref{WeNo3} as
\begin{align*}
&\|\partial_t  Z^{|a|} v\|^2_{L^2}\| \langle t\rangle\nabla Z^{[|a|/2]} d\|^2_{L^\infty(r\leq \langle t\rangle/2)}
+\| \partial_t Z^{[|a|/2]}v\|_{L^3}
\|\langle t\rangle \nabla Z^{|a|}d \|^2_{L^6(r\leq \langle t\rangle/2) } \\[-4mm]\\
&\lesssim E^v_{|a|+1}( X^d_{[|a|/2]+3}+E^d_{[|a|/2]+1} )
+ E^v_{[|a|/2]+2} (X^d_{|a|+2} +E^d_{|a|+1} ) .
\end{align*}
Therefore, we have
\begin{align*}
H_3&\lesssim E^v_{|a|+1}( X^d_{[|a|/2]+3}+E^d_{[|a|/2]+3} )
+ E^v_{[|a|/2]+3} (X^d_{|a|+2} +E^d_{|a|+1} ) .
\end{align*}

For $H_4$, $H_5$ and $H_6$, along the same line, we can obtain that
\begin{align*}
&H_4\lesssim E^v_{|a|+1}E^v_{[|a|/2]+3}(E^d_{[|a|/2]+4}+X^d_{[|a|/2]+4})+(X^d_{|a|+2}+E^d_{|a|+2}) (E^v_{[|a|/2]+3})^2,\\
&H_5\lesssim (E^v_{|a|}E^d_{[|a|/2]+3}+E^v_{[|a|/2]+2}E^d_{|a|+1})  E^v_{[|a|/2]+3}(E^d_{[|a|/2]+3}+X^d_{[|a|/2]+3}) (E^d_{[|a|/2]+3}+1),\\
&H_6\lesssim (E^v_{|a|}E^d_{[|a|/2]+3}+E^v_{[|a|/2]+2}E^d_{|a|+1})(E^d_{[|a|/2]+3}+X^d_{[|a|/2]+3}) (E^d_{[|a|/2]+3}+1).
\end{align*}

Combining the estimates of $H_1$, ..., $H_6$ gives the lemma.
\end{proof}

Now we show that $X^d_{\kappa}$ can be controlled by the generalized energy under certain small energy assumptions.
\begin{lem}\label{lemmaWeightedNorm}
Suppose $\kappa\geq 9$, $E^v_\kappa\lesssim \epsilon$ and $E^d_{\kappa-1}\lesssim \epsilon$.  Then there hold that
\begin{equation*}
X^d_{\kappa-1}\lesssim E^d_{\kappa-1}, \quad
X^d_{\kappa+1}\lesssim E^d_{\kappa+1}.
\end{equation*}
\end{lem}
\begin{proof}
Let $\kappa\geq 9$,  $|a|+2\leq \kappa+1$, then one has $[a/2]+4\leq \kappa-1.$
Before proving the lemma, we first show that under the  assumption of $E^v_{\kappa}\lesssim \epsilon$, $E^d_{\kappa-1}\lesssim \epsilon$,
there holds
\begin{align}\label{WeNo}
X^d_{|a|+2}
\lesssim& E^d_{|a|+2}+E^d_{|a|+1}(X^d_{[|a|/2]+3}+E^d_{[|a|/2]+3})  \nonumber\\[-4mm] \nonumber \\
&\qquad + E^v_{|a|+1} (E^d_{[|a|/2]+4}+X^d_{[|a|/2]+4})
 + (E^d_{|a|+2}+X^d_{|a|+2})E^v_{[|a|/2]+3}.
\end{align}
Actually, by Lemma \ref{WE-1}, one has
\begin{align*}
X^d_{|a|+2}&\lesssim E^d_{|a|+2}+\sum_{|b|\leq |a|}\| (t+r)f_{b}^2\|^2_{L^2}.
\end{align*}
On the other hand, thanks to Lemma \ref{lemWeNo} and the assumption $E^v_{\kappa}\lesssim \epsilon$, $E^d_{\kappa-1}\lesssim \epsilon$, one can easily check that
\begin{align*}
\| (t+r)f_{a}^2\|^2_{L^2}
&\lesssim E^d_{|a|+1}(X^d_{[|a|/2]+3}+E^d_{[|a|/2]+3})  \nonumber\\[-4mm] \nonumber \\
&\qquad + E^v_{|a|+1} (E^d_{[|a|/2]+4}+X^d_{[|a|/2]+4})
 + (E^d_{|a|+2}+X^d_{|a|+2})E^v_{[|a|/2]+3}.
\end{align*}
Thus \eqref{WeNo} can be deduced directly.

Now we turn to the proof of the lemma.
Let $|a|+2\leq \kappa-1$, one can get by\eqref{WeNo} that
\begin{align*}
X^d_{\kappa-1}
&\lesssim E^d_{\kappa-1}+E^d_{\kappa-1}(E^d_{\kappa-1}+E^v_{\kappa-1})
+(E^d_{\kappa-1}+E^v_{\kappa-1})X^d_{\kappa-1}.
\end{align*}
Using the assumption of $E^v_{\kappa}\lesssim \epsilon$ and $E^d_{\kappa-1}\lesssim \epsilon$ yields
\begin{equation*}
X^d_{\kappa-1}\lesssim E^d_{\kappa-1}.
\end{equation*}
Furthermore, for $|a|\leq\kappa-1$, we get from \eqref{WeNo} that
\begin{align*}
X^d_{\kappa+1} &\lesssim E^d_{\kappa+1}+E^d_{\kappa+1}(E^d_{\kappa-1}+E^v_{\kappa})
+(E^d_{\kappa-1}+E^v_{\kappa-1})X^d_{\kappa+1}+(E^d_{\kappa+1}+E^v_{\kappa})X^d_{\kappa-1},
\end{align*}
from which together with the assumption implies
\begin{equation*}
X^d_{\kappa+1}\lesssim E^d_{\kappa+1}.
\end{equation*}
Thus the lemma is proved.
\end{proof}

An immediate consequence of the weighted $L^2$ norm estimate is that we can gain more decay for the good unknowns
of the orientation field.
\begin{lem}\label{lemWeiGood}
Suppose $\kappa\geq 9$, $E^v_\kappa\lesssim \epsilon$, $E^d_{\kappa-1}\lesssim \epsilon$, then there hold
\begin{align*}
&\|\langle t\rangle| (\partial_t + \partial_r) \partial Z^a d\|_{L^2}
\lesssim (E^d_{\kappa-1})^{\frac12},\quad \forall\ |a|\leq \kappa-3, \\
&\|\langle t\rangle| (\partial_t + \partial_r) \partial Z^a d\|_{L^2}
\lesssim (E^d_{\kappa+1})^{\frac12},\quad
\forall\ |a|\leq \kappa-1.
\end{align*}
\end{lem}
\begin{proof}
It is a direct consequence of Lemma \ref{GoodDeri}, Lemma \ref{lemWeNo} and Lemma \ref{lemmaWeightedNorm}.
\end{proof}

Another consequence is that we can obtain the decay for $L^\infty$ norm of $\partial Z^{a} d$.
\begin{lem}\label{lemdecayliq}
Suppose $\kappa\geq 9$, $E^v_\kappa\lesssim \epsilon$, $E^d_{\kappa-1}\lesssim \epsilon$, then there hold
\begin{align*}
&\langle t\rangle\|\partial Z^{a} d\|_{L^\infty}\lesssim  (E^d_{\kappa-1})^{\frac12},\quad \forall\ |a|\leq \kappa-4, \\
&\langle t\rangle \|\partial Z^{a} d\|_{L^\infty}\lesssim (E^d_{\kappa+1})^{\frac12},\quad
\forall\ |a|\leq \kappa-2.
\end{align*}
\end{lem}
\begin{proof}
By \eqref{K-S-3D-2} and \eqref{KSWeiIn1}, one has
\begin{align*}
\|\partial Z^a d \|_{L^\infty}
&\leq \langle t\rangle^{-1} (\| r\partial Z^a d \|_{L^\infty(r\geq \langle t\rangle/2)}+\| \langle t\rangle\partial Z^a d \|_{L^\infty(r\leq \langle t\rangle/2)})\\
&\lesssim \langle t\rangle^{-1}\big[ (E^d_{|a|+3})^{\frac12}+(X^d_{|a|+3})^{\frac12}\big].
\end{align*}
Using Lemma \ref{lemmaWeightedNorm} we can obtain the result.
\end{proof}

\section{Decay estimates for the velocity}
This section is devoted to the decay estimate under the a priori estimate assumption of the generalized energy.
The estimate for the weighted $L^2$ norm in Section 3 is also applicable.

\begin{lem}\label{lemDecayVel}
Suppose $E^v_{\kappa}+\sum_{1\leq|a|\leq\kappa} \|\nabla^a v\|_{L^2_{t,x}} +E^d_{\kappa-1}\lesssim \epsilon$
with $\kappa\geq 9$, then there hold
\begin{align}
\| \langle t\rangle^{\frac34}v\|_{L^\infty_x}\lesssim& ~\epsilon^{\frac12},\label{decayvel1}\\
\| \langle t\rangle^{\frac54}\nabla v\|_{L^\infty_x}\lesssim&~ \epsilon^{\frac12},\label{decayvel2}\\
\|\langle t \rangle^{\frac32} \nabla^a v\|_{L^\infty_{x}} \lesssim&~ \epsilon^{\frac12} (\ln {\langle t\rangle})^{\frac12},
\quad \forall\  2\leq |a| \leq \kappa-3. \label{decayvel3}
\end{align}
\end{lem}
\begin{remark}
The first two decay estimates \eqref{decayvel1} and \eqref{decayvel2} are sharp in the sense that the decay rate is the same as the linear heat equation
if the initial data lies in the energy space. The restricted decay rate for higher order derivatives \eqref{decayvel3}  is due to the Ericksen stress.
\end{remark}
\begin{proof}
Thanks to the local well-posedness \cite{JL}, one easily has an uniform bound on the life span of lower-bound of $\delta^{-1}$ where $\delta$ is the size of the initial perturbation around equilibrium. Thus in the following argument, we always assume $t\geq 1$. Correspondingly, the $L^p_t$ norm denotes $L^p([1,t))$ for simplicity, where the integral time interval is $[1,t)$.

We begin by writing down the expression for velocity:
\begin{equation}\label{velmild}
v(t,x)=e^{t\Delta} v(1)-\int_1^t e^{(t-s)\Delta} \mathbb{P}\big[v\cdot\nabla v+\nabla\cdot(\nabla d\otimes\nabla d)\big](s) \rmd s,
\end{equation}
where $\mathbb{P}$ is the Leray projector.
The inhomogeneous term can be rewritten as
\begin{eqnarray*}
&&\int_1^t e^{(t-s)\Delta}\mathbb{P} \big[v\cdot\nabla v+\nabla\cdot(\nabla d\otimes\nabla d)\big](s) ds \\
&&=\int_1^t\int_{\mathbb{R}^3}
\nabla_y H(t-s,x-y) (v\otimes v+\nabla d\otimes\nabla d)(s,y) \rmd s\rmd y,
\end{eqnarray*}
where $H(t,x)$ is a function of the three dimensional heat kernel
$C t^{-\frac32} \exp(-|x|^2/t)$ convoluting the Leray projection operator.
Moreover, $H(t,x)$ behaves like (see for instance in
\cite{LR},
Proposition 11.1):
\begin{equation*}
t^{-\frac32} h (|x|/\sqrt t),
\end{equation*}
where
 $$|h(y)|\lesssim 1/{\langle y\rangle^{3}},\quad
|\nabla_y h(y)|\lesssim 1/{\langle y\rangle^{4}},\quad
|\nabla^2_y h(y)|\lesssim 1/{\langle y\rangle^{5}}.$$

Now we are ready to show the lemma. We first prove \eqref{decayvel1}.
To this end, 
we write
\begin{eqnarray*}
&& t^{\frac34} v=J_0
+J_1+J_2+J_3,
\end{eqnarray*}
where
\begin{eqnarray*}
&&J_0=t^{\frac34} e^{t\Delta} v(1),\\
&&J_1=t^{\frac34} \int_1^{\frac{t}{2}}\int_{\mathbb{R}^3}
 \nabla_y H(t-s,x-y) (v\otimes v+\nabla d\otimes\nabla d)(s,y)\  \rmd s\rmd y,\\
&&J_2=t^{\frac34} \int_{\frac{t}{2}}^{t}\int_{\mathbb{R}^3}
\nabla_y H(t-s,x-y)(v\otimes v)(s,y)\  \rmd s\rmd y,\\
&&J_3=t^{\frac34} \int_{\frac{t}{2}}^{t}\int_{\mathbb{R}^3}
\nabla_y H(t-s,x-y)(\nabla d\otimes\nabla d)(s,y)\  \rmd s\rmd y.
\end{eqnarray*}
For $J_0$,  one gets by Young's inequality that
\begin{equation*}
J_0\lesssim \|v(1,\cdot)\|_{L^2_x}\lesssim ({E_{\kappa}^v})^{\frac12}.
\end{equation*}
For $J_1$,  one has  $t/2 \leq t-s \leq t$ from $1\leq s\leq t/2$.
Hence applying Young's inequality 
and the Hardy-Littlewood-Sobolev inequality 
yields
\begin{align*}
J_1
&~\leq  2\int_1^{\frac{t}{2}} (t-s)^{\frac34}
\| \nabla H(t-s,\cdot )\|_{L^2_x} \big \| (|v|^2 +|\nabla d|^2 )(s,\cdot) \big\|_{L^2_x}\  \rmd s\\
&~ \lesssim \int_1^{\frac{t}{2}} (t-s)^{-\frac12}
\big\| (|v|^2 +|\nabla  d|^2)(s,\cdot) \big\|_{L^2_x}\  \rmd s\\
&~\lesssim  \big \||v|^2 +|\nabla  d|^2 \big\|_{L^2_{t,x}}.
\end{align*}
Then one can obtain the estimate for $J_1$ by using Lemma \ref{lemdecayliq}:
\begin{eqnarray*}
J_1\lesssim \|\nabla v\|_{L^2_{t,x}} (E^v_{\kappa})^{\frac12}+\| \langle t\rangle^{-1} \|_{L^2_t} E^d_{\kappa-1}
\lesssim \|\nabla v\|^2_{L^2_{t,x}}+E^v_{\kappa}+E^d_{\kappa-1}.
\end{eqnarray*}
For the term $J_2$, one has $t/2\leq s\leq t$. Hence, by
Young's inequality and the {Hardy-Littlewood-Sobolev} inequality, we can derive that
\begin{align*}
J_2&~\leq \int_{\frac t2}^t s^{\frac34}|\nabla H(t-s)|*  |(v\otimes v )(s)|   \ \rmd s \\
&~\lesssim \int_{\frac t2}^t \|\nabla H(t-s)\|_{L^1_x} \| v(s)\|_{L^\infty_x} \| s^{\frac34} v(s)\|_{L^\infty_x}\  \rmd s \\
&~\lesssim \int_{\frac t2}^t (t-s)^{-\frac12} \| v(s)\|_{L^\infty_x} \| s^{\frac34} v(s)\|_{L^\infty_x}\  \rmd s \\
&~\lesssim \| v(s)\|_{L^2_t L^\infty_{x}} \| s^{\frac34} v(s)\|_{L^\infty_{t,x}}\\[-4mm]\\
&~ \lesssim \sum_{1\leq |a|\leq 2}\|\nabla^a v\|_{L^2_tL^2_x}
\| s^{\frac34} v(s)\|_{L^\infty_{t,x}}.
\end{align*}
For $J_3$, similar to the estimate of $J_2$, one has
\begin{equation*}
J_3\lesssim  \| s^{\frac34} \nabla d(s)\|_{L_t^\infty L^\infty_x} \| \nabla d(s)\|_{L_t^2 L^\infty_x}\lesssim E^d_{\kappa-1},
\end{equation*}
where we used Lemma \ref{lemdecayliq} in the last estimate.
Gathering the estimates for $J_0$, $J_1$, $J_2$, $J_3$, we conclude
\begin{align*}
t^{\frac34}v(t)\lesssim ({E_{\kappa}^v})^{\frac12}+E^v_{\kappa}+E^d_{\kappa-1}+\|\nabla v\|^2_{L^2_{t,x}}
+\sum_{1\leq |a|\leq 2}\|\nabla^a v\|_{L^2_tL^2_x}
\| s^{\frac34} v(s)\|_{L^\infty_{t,x}}.
\end{align*}
Absorbing the last term yields \eqref{decayvel1}.

Secondly, we treat \eqref{decayvel2}.
Similar to the estimate of \eqref{decayvel1}, we write
\begin{equation*}
t^{\frac54}\nabla v(t,x)=J^1_0+J^1_1+J^1_2+J^1_3,
\end{equation*}
where
\begin{eqnarray*}
&&J^1_0=t^{\frac54} \nabla e^{t\Delta} v(1),\\
&&J^1_1=t^{\frac54} \int_1^{\frac{t}{2}}\int_{\mathbb{R}^3}
 \nabla^2_y H(t-s,x-y) (v\otimes v+\nabla d\otimes\nabla d)(s,y)\  \rmd s\rmd y,\\
&&J^1_2=t^{\frac54} \int_{\frac{t}{2}}^{t}\int_{\mathbb{R}^3}
\nabla_y H(t-s,x-y) \nabla(v\otimes v)(s,y)\  \rmd s\rmd y,\\
&&J^1_3=t^{\frac54}\int_{\frac{t}{2}}^{t}\int_{\mathbb{R}^3}
\nabla_y H(t-s,x-y) \nabla(\nabla d\otimes\nabla d)(s,y)\  \rmd s\rmd y.
\end{eqnarray*}
We remark that the formulation of $J_1^1$ is different from $J_1$, while the other terms are similar.

For $J_0^1$, we get by Young's inequality that
\begin{equation*}
J^1_0\lesssim \|v(1,\cdot)\|_{L^2_x}\lesssim ({E_{\kappa}^v})^{\frac12}.
\end{equation*}
Applying Young's inequality and the Hardy-Littlewood-Sobolev inequality, we have that
\begin{align*}
J^1_1
&\leq  2\int_1^{\frac{t}{2}} (t-s)^{\frac54}
\big\| \nabla^2 H(t-s,\cdot )\big\|_{L^2_x} \big \| (| v|^2 +|\nabla d|^2)(s,\cdot) \big\|_{L^2_x}\  \rmd s\\
&\lesssim \big \| |v|^2 +|\nabla d|^2\big\|_{L^2_{t,x}}
\lesssim \|\nabla v\|^2_{L^2_{t,x}}+E^v_{\kappa}+E^d_{\kappa-1},
\end{align*}
where the relation $t/2 \leq t-s \leq t$ has been used.
For $J_2^1$, we derive that
\begin{eqnarray*}
&&J^1_2
\lesssim \sum_{1\leq |a|\leq 2}\|\nabla^a v\|_{L^2_tL^2_x} \| s^{\frac54}\nabla v(s)\|_{L^\infty_{t,x}}.
\end{eqnarray*}
The estimate for $J_3^1$ is similar to $J_2$ with slight modifications.
By Young's inequality, the {Hardy-Littlewood-Sobolev} inequality and Lemma \ref{lemdecayliq}, one derives that
\begin{align*}
J^1_3&\lesssim \int_{\frac t2}^t \|\nabla H(t-s)\|_{L^1_x} \| s^{\frac14}  \nabla^2 d(s)\|_{L^\infty_x} \| s \nabla d(s)\|_{L^\infty_x}\  \rmd s \\
& \lesssim   \|s^{-\frac34} \|_{L^2_t} \| s \nabla^2 d(s)\|_{L^\infty_{t,x}} \|s \nabla d(s)\|_{L^\infty_{t,x}}\lesssim E_{\kappa-1}^d.
\end{align*}
Gathering the estimate for $J_0^1$, $J_1^1$, $J_2^1$, $J_3^1$, we conclude
\begin{align*}
t^{\frac54}v(t)\lesssim ({E_{\kappa}^v})^{\frac12}+E^v_{\kappa}+E^d_{\kappa-1}+\|\nabla v\|^2_{L^2_{t,x}}
+\sum_{1\leq |a|\leq 2}\|\nabla^a v\|_{L^2_tL^2_x}
\| s^{\frac54} v(s)\|_{L^\infty_{t,x}}.
\end{align*}
Absorbing the last term yields \eqref{decayvel2}.

Finally, we treat \eqref{decayvel3}. For simplicity of presentation, we only show the case for $|a|=2$. The higher-order case can be estimated in the same but lengthier argument.

Similar to the estimate of \eqref{decayvel2}, we write:
\begin{equation*}
t^{\frac32}\nabla^2 v(t,x)=J^2_0+J^2_1+J^2_2+J^2_3,
\end{equation*}
where
\begin{eqnarray*}
&&J^2_0=t^{\frac32}e^{t\Delta} \nabla^2 v(1),\\
&&J^2_1=t^{\frac32} \int_1^{\frac{t}{2}}\int_{\mathbb{R}^3}
 \nabla^2_y H(t-s,x-y) \nabla(v\otimes v+\nabla d\otimes\nabla d)(s,y)\  \rmd s\rmd y,\\
&&J^2_2=t^{\frac32} \int_{\frac{t}{2}}^{t}\int_{\mathbb{R}^3}
\nabla_y H(t-s,x-y) \nabla^2(v\otimes v)(s,y)\  \rmd s\rmd y,\\
&&J^2_3=t^{\frac32}\int_{\frac{t}{2}}^{t}\int_{\mathbb{R}^3}
\nabla_y H(t-s,x-y) \nabla^2(\nabla d\otimes\nabla d)(s,y)\  \rmd s\rmd y.
\end{eqnarray*}
As the estimates of $J_0^2$, $J_1^2$ and $J_2^2$ are similar to those of $J_0^1$, $J_1^1$ and $J_2^1$ respectively,
we only sketch them here. By Young's inequality, the {Hardy-Littlewood-Sobolev} inequality and Lemma \ref{lemdecayliq}, one has
\begin{align*}
J^2_0&\lesssim t^{-\frac14}\|v(1,\cdot)\|_{L^2_x}\lesssim ({E_{\kappa}^v})^{\frac12},\\[-4mm]\\
&J^2_1\lesssim \|\nabla(v\otimes v+\nabla d\otimes\nabla d)\|_{L_t^{\frac43}L^2_x}
\lesssim \sum_{1\leq |a|\leq \kappa}\|\nabla^a v\|^2_{L^2_{t,x}}+E^v_{\kappa}+E^d_{\kappa-1},\\
&J^2_2\lesssim  \|s^{\frac32}\nabla^2(v\otimes v)\|_{L^2_tL^\infty_x}\\[-4mm] \\
&\quad \lesssim \sum_{1\leq |a|\leq 2}\|\nabla^a v\|_{L^2_tL^2_x} \|s^{\frac32}\nabla^2 v\|_{L^\infty_{t,x}}
+\|s^{\frac54}\nabla v\|^2_{L^\infty_{t,x}} \| s^{-1}\|_{L^2_t}.
\end{align*}
For $J_3^2$, by Young's inequality, the {Hardy-Littlewood-Sobolev} inequality, one derives that
\begin{align*}
J^2_3&\lesssim \int_{\frac t2}^t \|\nabla_y H(t-s)\|_{L^{p'}_x} \sum_{2\leq |a|\leq 3}\| s^{\frac12}  \nabla^a d(s)\|_{L^{p}_x} \sum_{1\leq |a|\leq 2}\| s \nabla^a d(s)\|_{L^\infty_x}\  \rmd s \\
& \lesssim  \int_{\frac t2}^t  (t-s)^{-\frac12-\frac{3}{2p}} \sum_{2\leq |a|\leq 3}\| s^{\frac12} \nabla^a d(s)\|_{L^p_x} \sum_{1\leq |a|\leq 2}\|s \nabla^a d(s)\|_{L^\infty_x}\  \rmd s \\
&\lesssim   \sum_{2\leq |a|\leq 3} \| s^{\frac12} \nabla^a d(s)\|_{L^{\frac{2p}{p-3}}_t L^p_x}
\sum_{1\leq |a|\leq 2}\|s \nabla^a d(s)\|_{L^\infty_{t,x}},
\end{align*}
where $p$ and $p'$ are dual index, $3<p\leq \infty$. Here, the constraint $3<p$ is due to the application of the {Hardy-Littlewood-Sobolev} inequality.
To earn the maximum decay, we take $p=\infty$. Consequently,
by Lemma \ref{lemdecayliq}, there holds
\begin{equation*}
J_3^2 \lesssim (E^d_{\kappa-1})^{\frac12} \|s^{-\frac12}\|_{L^2_t} \sum_{1\leq |a|\leq 3}\|s \nabla^a d(s)\|_{L^\infty_{t,x}}
\lesssim (\ln \langle t\rangle)^{\frac12}  E^d_{\kappa-1}.
\end{equation*}
Finally, one conclude that
\begin{align*}
t^{\frac32}\nabla^2 v(t)\lesssim& ({E_{\kappa}^v})^{\frac12}+E^v_{\kappa}
+(\ln \langle t\rangle)^{\frac12}E^d_{\kappa-1}
+\sum_{1\leq |a|\leq \kappa}\|\nabla^a v\|^2_{L^2_{t,x}}
+\| s^{\frac54} \nabla v(s)\|^2_{L^\infty_{t,x}}\\
&+\sum_{1\leq |a|\leq 2}\|\nabla^a v\|_{L^2_tL^2_x} \|s^{\frac32}\nabla^2 v\|_{L^\infty_{t,x}}.
\end{align*}
Absorbing the last term, \eqref{decayvel3} can be inferred from \eqref{decayvel2}.
\end{proof}

An immediate consequence of the above lemma is the decay estimate for $\partial_t v$.
\begin{lem}
Under the assumption of Lemma \ref{lemDecayVel}, there holds
\begin{align*}
\|\partial_t v\|_{L^\infty}\lesssim \epsilon^{\frac12} \langle t\rangle^{-\frac32} (\ln \langle t \rangle)^{\frac12}.
\end{align*}
\end{lem}
\begin{proof}
We first show that (this decay rate is not optimal but it is sufficient for our purpose)
\begin{equation}\label{decaypre}
\|\nabla p\|_{L^\infty}\lesssim \epsilon\langle t\rangle^{-\frac53}.
\end{equation}
Due the incompressible condition for $v$, we can use the Leray projector to write the pressure explicitly:
\begin{align*}
\nabla p={(\mathbb{P}-1)}(v\nabla v+\nabla\cdot(\nabla d \otimes\nabla d)).
\end{align*}
Then by Sobolev imbedding, one deduces that
\begin{align*}
\|\nabla p\|_{L^\infty}
&\lesssim \|\nabla p\|^{\frac12}_{L^6}\|\nabla^2 p\|^{\frac12}_{L^6}\\[-4mm]\\
&\lesssim \sum_{|a|\leq 1}\|\nabla^a (v\nabla v+\nabla\cdot(\nabla d \otimes\nabla d)) \|_{L^6}\\
&\lesssim \sum_{|a|\leq 1}(\|\nabla^a v\|^{\frac13}_{L^2} \|\nabla^a v\|^{\frac23}_{L^\infty})
\sum_{1\leq |a|\leq 2}\|\nabla^a v\|_{L^\infty} \\
&\quad + \sum_{|a|\leq 1}(\| \nabla \nabla^a d\|^{\frac13}_{L^2}\| \nabla \nabla^a d\|^{\frac23}_{L^\infty})
\sum_{1\leq |a|\leq 2}\| \nabla \nabla^a d\|_{L^\infty}.
\end{align*}
all of which can be controlled by $\epsilon \langle t\rangle^{-\frac53}$ by Lemma \ref{lemdecayliq} and Lemma \ref{lemDecayVel}.  Thus \eqref{decaypre} is proved.

Then, thanks to Lemma \ref{lemdecayliq},  Lemma \ref{lemDecayVel} and \eqref{decaypre},  we deduce that
\begin{align*}
\|\partial_t v\|_{L^\infty}&\lesssim \|\Delta v\|_{L^\infty}+\|v\cdot \nabla v \|_{L^\infty}+
\| \nabla\cdot(\nabla d\otimes \nabla d)\|_{L^\infty}+\|\nabla p\|_{L^\infty}\\
&\lesssim \langle t\rangle^{-\frac32} (\ln \langle t \rangle)^{\frac12}\epsilon^{\frac12},
\end{align*}
which yields the lemma.
\end{proof}

\section{The energy estimates}
This section is devoted to the energy estimates, which corresponds to generalized energy estimate for the velocity,
the higher-order and the lower-order generalized energy estimates for the orientation field.

\subsection{Generalized energy estimate for velocity}
In this subsection, we estimate the generalized energy for the velocity, which turns out to be uniformly bounded in time. The main trouble in the estimate of $E^v_{\kappa}$
is due to commutators between the scaling operator and the viscosity terms. Fortunately, we can take the  approach borrowed from \cite{CLLM, Kessenich}.

Let $\kappa \geq 9$, $0\leq|a|\leq \kappa$.
Taking the $L^2$ inner product of $\eqref{LiquidHyper_GeDe}_1$  with $Z^a v$, we have
\begin{equation}\label{HighEnery_v}
\frac 12 \frac{d}{dt}
\int_{\mathbb{R}^n}|Z^a v|^2\ \rmd x
-\int_{\mathbb{R}^n} \mu \Delta (S-1)^{a_1}\Gamma^{a'} v \cdot  Z^a v\ \rmd x
=\int_{\mathbb{R}^n}
f^1_{ a} \cdot  Z^a v  \ \rmd x.
\end{equation}
Recalling the expression for $f^1_a$ in \eqref{f1f2},  one has
\begin{align*}
\int_{\mathbb{R}^n}
f^1_{\alpha a} \cdot  Z^a v  \ \rmd x
\lesssim\sum_{|b|+|c|\leq |a|}
\| \nabla Z^a v\|_{L^2}
\big(\big\| |\nabla Z^b d|\
|\nabla Z^c d| \big\|_{L^2}
+\big\| |Z^b v|\ | Z^c v| \big\|_{L^2}\big).
\end{align*}
Due to the symmetry between $b$ and $c$, we assume $|b|\leq |c|$ without loss of generality. Thus one has $|c|\leq [{|a|}/2]\leq \kappa-4$. Consequently, by Lemma \ref{lemdecayliq},
the above can be further bounded by
\begin{align*}
&\| \nabla Z^a v\|_{L^2}
(\|\nabla Z^{|a|} d\|_{L^2}\|\nabla Z^{[{|a|}/2]} d \|_{L^\infty}
+\| Z^{|a|}v \|_{L^2}\| Z^{[{|a|}/2]} v \|_{L^\infty}
)\nonumber \\[-4mm] \nonumber\\
&\lesssim \| \nabla Z^a v\|_{L^2} ( \langle t\rangle^{-1}(E^d_{\kappa+1} E^d_{\kappa-1})^\frac12
+\| \nabla Z^\kappa v\|_{L^2} \| Z^\kappa v\|_{L^2}   ).
\end{align*}

Next, we estimate the diffusion terms with indefinite sign.
To this end, we need a technical lemma followed from \cite{CLLM}.

\begin{lem}\cite{CLLM} \label{lemItera}(Iteration lemma)
Let $\{f_l\},\ \{g_l\},\ \{F_l\}$ be three nonnegative sequences,
where $0\leq l\leq \kappa$.
Suppose that
\begin{equation*}
f_0+g_0\lesssim F_0,
\end{equation*}
and for all $1\leq l\leq \kappa$,
\begin{equation*}
f_l+g_l-g_{l-1}\lesssim F_l.
\end{equation*}
Then there holds
\begin{equation*}
\sum_{0\leq m\leq l}(f_m+g_m)\lesssim \sum_{0\leq m\leq l}F_m,
\end{equation*}
for all $0\leq l\leq \kappa$.
\end{lem}

Now we are ready to  estimate the diffusion terms as follows:
\begin{align*}
&-\mu \int_{\mathbb{R}^n} \Delta (S-1)^{a_1}\Gamma^{a'} v \cdot  S^{a_1}\Gamma^{a'} v\ \rmd x \nonumber\\
&=-\mu \sum_{l\leq a_1} C_{a_1}^l (-1)^{a_1-l} \int_{\mathbb{R}^n}\Delta S^l\Gamma^{a'} v \cdot  S^{a_1}\Gamma^{a'} v\ \rmd x \nonumber\\
&\geq -\frac12\mu \|\nabla S^{a_1}\Gamma^{a'} v \|^2_{L^2}
-  \mu C\sum_{l\leq a_1-1}
\|\nabla S^l\Gamma^{a'} v \|^2_{L^2}.
\end{align*}
Inserting the above into \eqref{HighEnery_v}, together with the estimates of the nonlinear terms yield
\begin{eqnarray*}
&&\frac 12 \frac{d}{dt}
\int_{\mathbb{R}^n}|Z^a v|^2 \ \rmd x
+\frac12 \mu \|\nabla Z^a v \|^2_{L^2_x}
- \mu C\sum_{l\leq a_1-1}\|\nabla S^l\Gamma^{a'} v \|^2_{L^2_x} \\
&&\leq C\| \nabla Z^\kappa v\|^2_{L^2_x} \| Z^\kappa v\|^2_{L^2_x}+C
\langle t\rangle^{-2} E^d_{\kappa+1} E^d_{\kappa-1}.
\end{eqnarray*}
Then integrating in time over $[0,t)$ on both sides of the above inequality  gives
\begin{eqnarray*}
&& \| Z^a v(t,\cdot)\|^2_{L^2_x}
+  \mu \|\nabla Z^a v \|^2_{L^2_tL^2_x}
-\mu C\sum_{l\leq {a_1}-1} \|\nabla S^l\Gamma^{a'} v\|^2_{L^2_tL^2_x} \\
&&\lesssim \| Z v(0,\cdot)\|^2_{L^2_x}+\int_0^t\langle \tau \rangle^{-2} E^d_{\kappa+1}(\tau) E^d_{\kappa-1}(\tau) \ \rmd \tau+
 \|\nabla Z^{\kappa} v \|^2_{L^2_tL^2_x}
 \|Z^{\kappa} v \|^2_{L^\infty_tL^2_x}
. \nonumber
\end{eqnarray*}
Now we can use Lemma \ref{lemItera} to absorb the lower order diffusion terms to derive that
\begin{eqnarray*}
&& \| Z^{\kappa} v(t,\cdot)\|^2_{L^2_x}
+ \mu \|\nabla Z^{\kappa} v \|^2_{L^2_tL^2_x} \\
&&\lesssim \| Z^{\kappa} v(0,\cdot)\|^2_{L^2_x}+\int_0^t\langle \tau \rangle^{-2} E^d_{\kappa+1}(\tau) E^d_{\kappa-1}(\tau) \ \rmd \tau+
 \|\nabla Z^{\kappa} v(\cdot,\cdot) \|^2_{L^2_tL^2_x}
 \|Z^{\kappa} v(\cdot,\cdot) \|^2_{L^\infty_tL^2_x}
. \nonumber
\end{eqnarray*}
This gives the a priori estimate \eqref{priorivel}.

\subsection{Higher-order energy estimate for the orientation field}
In this subsection, we estimate the higher-order energy for the orientation field $d$, which will exhibit some polynomial growth in time. The main difficulty
is the potential derivative loss due to the quasilinear effect. Fortunately, this difficulty can be overcome by a symmetry structure of the system. See the estimate of $I_{12}$ below.

Let $\kappa \geq 9$, $0 \leq|a|\leq \kappa$.
Taking the $L^2$ inner product of  ${\eqref{LiquidHyper_GeDe}}_2$ with $\partial_t Z^a d$ gives
\begin{eqnarray}\label{HighEnergyVe}
\frac 12 \frac{d}{dt}
\int_{\mathbb{R}^n}|\partial Z^a d|^2\ \rmd x
=\int_{\mathbb{R}^n}
f^2_{a} \cdot  \partial_t Z^a d  \ \rmd x
. \nonumber
\end{eqnarray}
Thanks to \eqref{f1f2}, the right hand side above can be divided into $I_1$ and $I_2$, where
\begin{align*}
I_{1}=&-\int_{\mathbb{R}^n}
\big[ v\cdot\nabla(  v\cdot\nabla Z^a d)
+ 2 v\cdot\nabla\partial_t Z^a d\big]
 \cdot \partial_t Z^a d \ \rmd x,  \nonumber\\
&-\int_{\mathbb{R}^n}
(\partial_t Z^a v\cdot\nabla)  d
 \cdot \partial_t Z^a d \ \rmd x\\
\triangleq& I_{11}+I_{12}
\end{align*}
contains the highest order terms which may lose one derivative at first glance, and $I_2$ refers to the lower order ones:
\begin{align*}
I_2
=&-\sum_{{\substack{b+c=a\\c\neq a}}} C_a^b\int_{\mathbb{R}^3}
2Z^b v\cdot\nabla\partial_t Z^c d
\cdot \partial_t Z^a d\ \rmd x\\
&-\int_{\mathbb{R}^3}\sum_{{\substack{b+c=a\\b\neq a}}} C_a^b
(\partial_t Z^b v\cdot\nabla Z^c d) \cdot \partial_tZ^a d\ \rmd x\\
&+\int_{\mathbb{R}^3}\sum_{b+c+e=a} C_a^{b,c}
(\nabla Z^b d\cdot\nabla Z^c d-\partial_t Z^b d\cdot \partial_t Z^c d) Z^e d \cdot \partial_tZ^a d\ \rmd x\\
&-\int_{\mathbb{R}^3}\sum_{b+c+e+f=a} C_a^{b,c,e} \big[2 (Z^b v\cdot\nabla) Z^c d\cdot\partial_t Z^e d \big] Z^f d\cdot \partial_tZ^a d\ \rmd x\\
&-\int_{\mathbb{R}^3}\sum_{b+c+e+f+g=a}C_a^{b,c,e,f}
\big[  ( Z^b v\cdot \nabla) Z^c d \cdot ((Z^e v\cdot \nabla) Z^f d)  \big] Z^g d\cdot \partial_tZ^a d\ \rmd x \\
&-\int_{\mathbb{R}^3}
\sum_{{\substack{b+c+e=a\\e\neq a}}}C_a^{b,c}Z^b v\cdot\nabla( Z^c v\cdot \nabla Z^e d)\cdot \partial_tZ^a d\ \rmd x \\
\triangleq&I_{21}+I_{22}+I_{23}+I_{24}+I_{25}+I_{26}.
\end{align*}

Now we take care of the nonlinear terms group by group.

\subsubsection{Estimates of $I_2$}

\textbf{Estimate of $I_{21}$}:
\begin{align*}
I_{21}=&-\sum_{\substack{ b+c=a\\c\neq a}} C_a^b\int_{\mathbb{R}^3}
2Z^b v\cdot\nabla\partial_t Z^c d
\cdot \partial_t Z^a d\ \rmd x.
\end{align*}
If $|c|\leq |b|$, by Lemma \ref{lemdecayliq}, the above can be bounded by
\begin{align*}
&\tiny\sum_{{\substack{b+c=a\\|c|\leq |b|}}}\tiny
\|Z^b v\|_{L^2} \|\nabla\partial_t Z^c d\|_{L^\infty}\|\partial_t Z^a d\|_{L^2}
\lesssim \langle t\rangle^{-1} (E^v_{\kappa}E^d_{\kappa+1}E^d_{\kappa-1})^{\frac12}.
\end{align*}
Otherwise, if $|b|\leq |c|$, by \eqref{K-S-3D-2} and Lemma \ref{lemmaWeightedNorm}, one has
\begin{align*}
&\sum_{\tiny{\substack{|b|\leq[|a|/2]\\|c|\leq |a|-1}}} (\int_{r\geq\langle t\rangle/2 }+\int_{r\leq\langle t\rangle/2 })
|Z^b v|\ |\nabla\partial_t Z^c d|
\|\partial_t Z^a d|\ \rmd x\\
&\lesssim \langle t\rangle^{-1}\|r Z^{[|a|/2]} v\|_{L^\infty(r\geq\langle t\rangle/2)} \|\nabla\partial_t Z^{|a|-1} d\|_{L^2}
\|\partial_t Z^a d\|_{L^2} \\[-4mm]\\
&\quad+\langle t\rangle^{-1}\| Z^{[|a|/2]} v\|_{L^\infty} \|\langle t-r\rangle \nabla\partial_t Z^{|a|-1} d\|_{L^2(r\leq\langle t\rangle/2)}
\|\partial_t Z^a d\|_{L^2} \\[-4mm]\\
&\lesssim \langle t\rangle^{-1} (E^v_{\kappa})^{\frac12}E^d_{\kappa+1}.
\end{align*}

\textbf{Estimate of $I_{22}$}:
\begin{align*}
I_{22}=&-\sum_{\tiny\substack{b+c=a \\b\neq a }} C_a^b\int_{\mathbb{R}^3}
\partial_t Z^b v\cdot\nabla Z^c d
\cdot \partial_t Z^a d\ \rmd x.
\end{align*}
First for the case of $c=a$, the above quantity becomes
\begin{align*}
&-\int_{\mathbb{R}^3}
(\partial_t  v\cdot\nabla) Z^a d
\cdot \partial_t Z^a d\ \rmd x
\lesssim \|\partial_tv\|_{L^\infty} \| \partial Z^a d \|^2_{L^2}.
\end{align*}
Next, if $|b|\leq|c|\leq |a|-1$, by \eqref{KSWeiIn2}, Lemma \ref{lemmaWeightedNorm} and \eqref{K-S-3D-2}, one has
\begin{align*}
&-\sum_{\tiny\substack{b+c=a \\|b|\leq|c|\leq |a|-1 }}
C_a^b(\int_{r\leq\langle t\rangle/2}+\int_{r\geq\langle t\rangle/2})
\partial_t Z^b v\cdot\nabla Z^c d
\cdot \partial_t Z^a d\ \rmd x\\
&\lesssim \langle t\rangle^{-1}\sum_{\tiny\substack{b+c=a \\|b|\leq|c|\leq |a|-1 }}
\|\partial_t Z^b v\|_{L^3} \|\langle t-r\rangle\nabla Z^c d\|_{L^6(r\leq \langle t\rangle/2)}
\|\partial_t Z^a d\|_{L^2}\\
&\quad+\langle t\rangle^{-1}\sum_{\tiny\substack{b+c=a \\|b|\leq|c|\leq |a|-1 }}
\|r\partial_t Z^b v\|_{L^\infty(r\geq \langle t\rangle/2)} \|\nabla Z^c d\|_{L^2}
\|\partial_t Z^a d\|_{L^2}\\
&\lesssim \langle t\rangle^{-1} (E^v_{\kappa})^{\frac12}E^d_{\kappa+1}.
\end{align*}
Otherwise, if $|c|\leq |b|$, by Lemma \ref{lemdecayliq},
we have the control of
\begin{align*}
\|\partial_t Z^{|a|-1} v\|_{L^2} \|\nabla Z^{[|a|/2]} d\|_{L^\infty} \|\partial_t Z^a d\|_{L^2}
\lesssim  \langle t\rangle^{-1} (E^v_{\kappa})^{\frac12}E^d_{\kappa+1}.
\end{align*}

\textbf{Estimate of $I_{23}$}:
\begin{align*}
I_{23}&=\int_{\mathbb{R}^3} \sum_{b+c+e=a} C_a^{b,c}
(\nabla Z^b d\cdot\nabla Z^c d-\partial_t Z^b d\cdot \partial_t Z^c d) Z^e d \cdot\partial_tZ^a d\ \rmd x \\[-4mm]\\
&\lesssim\sum_{b+c+e=a}\int_{\mathbb{R}^3}
|\partial Z^b d|\ |\partial Z^c d|\ |Z^e d| |\partial_tZ^a d|\ \rmd x.
\end{align*}
By Lemma \ref{lemdecayliq}, the above  can be further bounded by
\begin{align*}
&\sum_{b+c+e=a}\int_{\mathbb{R}^3}
|\partial Z^b d|\ |\partial Z^c d|\ |Z^e d| |\partial_tZ^a d|\ \rmd x\\[-4mm]\\
&\lesssim\sum_{ [{|a|}/2] \leq |e|\leq |a|}
\|\partial Z^{[{|a|}/2]} d\|_{L^\infty} \|\partial Z^{[|a|/2]} d\|_{L^3} \|Z^e d\|_{L^6} \|\partial_tZ^a d\|_{L^2}\\[-4mm]\\
&\qquad+\sum_{|e|\leq[{|a|}/2]} \|\partial Z^{[|a|/2]} d\|_{L^\infty} \|\partial Z^{|a|} d\|_{L^2} \|Z^e d\|_{L^\infty} \|\partial_tZ^a d\|_{L^2}\\[-4mm]\\
&\lesssim\langle t\rangle^{-1} E^d_{\kappa+1} E^d_{\kappa-1}.
\end{align*}

\textbf{Estimate of $I_{24}$}:
\begin{align*}
I_{24}=&-\int_{\mathbb{R}^3} \sum_{b+c+e+f=a} C_a^{b,c,e} \big[2 (Z^b v\cdot\nabla) Z^c d\cdot\partial_t Z^e d \big] Z^f d\cdot\partial_tZ^a d\ \rmd x\\[-4mm]\\
&\lesssim \int_{\mathbb{R}^3} \sum_{b+c+e+f=a}  |Z^b v|\ |\partial Z^c d|\ |\partial Z^e d|\ |Z^f d| |\partial_tZ^a d|\ \rmd x.
\end{align*}
If $|f|\geq [|a|/2]$, by Sobolev inequalities and Lemma \ref{lemdecayliq}, the above can controlled by
\begin{align*}
&\sum_{[|a|/2]\leq |f|\leq |a|}\|Z^{[|a|/2]} v\|_{L^\infty} \|\partial Z^{[|a|/2]} d\|_{L^\infty}\ \|\partial Z^{[|a|/2]} d\|_{L^3} \|Z^f d\|_{L^6} \|\partial_tZ^a d\|_{L^2}\\
&\lesssim \langle t\rangle^{-\frac32}\|\nabla Z^{\kappa-1} v\|_{L^2}
E^d_{\kappa+1}E^d_{\kappa-1}.
\end{align*}
Otherwise, if $|f|\leq [|a|/2]$,  by Lemma \ref{lemdecayliq}, $I_{34}$ can be controlled by
\begin{align*}
&\sum_{|f|\leq [|a|/2]} \|Z^{|a|/2} v\|_{L^\infty} \|\partial Z^{|a|} d\|_{L^2} \|\partial Z^{[|a|/2]} d\|_{L^\infty}
\|Z^f d\|_{L^\infty} \|\partial_tZ^a d\|_{L^2}\\
&+\sum_{|f|\leq [|a|/2]} \|Z^{|a|} v\|_{L^2} \|\partial Z^{[a/2]} d\|_{L^\infty} \|\partial Z^{[|a|/2]} d\|_{L^\infty}
\|Z^f d\|_{L^\infty} \|\partial_tZ^a d\|_{L^2}\\
&\lesssim \langle t\rangle^{-1}
E^d_{\kappa+1}(E^v_{\kappa} E^d_{\kappa-1})^{\frac12}
[1+(E^d_{\kappa-1})^{\frac12}].
\end{align*}

\textbf{Estimate of $I_{25}$}:
By Lemma \ref{lemdecayliq}, we get
\begin{align*}
I_{25}=&-\sum_{b+c+e+f+g=a}C_a^{b,c,e,f}\int_{\mathbb{R}^3}
\big[  ( Z^b v\cdot \nabla) Z^c d \cdot ((Z^e v\cdot \nabla) Z^f d)  \big] Z^g d \cdot \partial_t Z^a d\ \rmd x\\
\lesssim&\sum_{[|a|/2]\leq |g|\leq |a|}
\|Z^{[|a|/2]} v\|^2_{L^6}
\| Z^{[|a|/2]} d\|^2_{L^\infty}  \|Z^g d\|_{L^6} \|\partial_t Z^a d\|_{L^2}\\
&+\sum_{|g|\leq [|a|/2] }
\big(\|Z^{|a|} v\|_{L^2}\|Z^{[|a|/2]} v\|_{L^\infty}
\| \nabla Z^{[|a|/2]} d\|^2_{L^\infty}  \\
&\qquad+\|Z^{[|a|/2]} v\|^2_{L^\infty}
\|\nabla Z^{|a|} d\|_{L^2} \| \nabla Z^{[|a|/2]} d\|_{L^\infty}\big)\|Z^g d\|_{L^\infty} \|\partial_t Z^a d\|_{L^2}\\[-4mm]\\
\lesssim& \langle t\rangle^{-1} \| \nabla Z^{\kappa-1 }v \|_{L^2} (E^v_{\kappa})^{\frac12} E^d_{\kappa+1}(1+(E^d_{\kappa-1})^{\frac12}).
\end{align*}

\textbf{Estimate of $I_{26}$}:
By the H\"{o}lder inequality, one has
\begin{align*}
I_{26}=-\sum_{{\substack{b+c+e=a\\e\neq a}}}C_a^{b,c}\int_{\mathbb{R}^3}
Z^b v\cdot\nabla( Z^c v\cdot \nabla Z^e d)
\cdot \partial_t Z^a d\ \rmd x
\lesssim \|\nabla Z^{\kappa} v\|^2_{L^2} E^d_{\kappa+1}.
\end{align*}

\subsubsection{Estimate of $I_{1}$}
We estimate $I_1$ in this part.

\textbf{Estimate of $I_{11}$}:  Employing integration by parts, one has
\begin{align*}
I_{11}&=-\int_{\mathbb{R}^n}
\big[ v\cdot\nabla(  v\cdot\nabla Z^a d)\big] \cdot \partial_t Z^a d \ \rmd x  \nonumber\\
&=-\int_{\mathbb{R}^n} (\partial_tv\cdot\nabla) Z^a d \cdot (v\cdot\nabla) Z^a d \ \rmd x
+\frac12\partial_t\int_{\mathbb{R}^n}|(v\cdot\nabla) Z^a d|^2 \ \rmd x \\
&\leq \|\nabla Z^{\kappa-1} v \|^2_{L^2_x} E^{d}_{\kappa+1}
+\frac12 \partial_t\int_{\mathbb{R}^n} |(v\cdot\nabla) Z^a d|^2 \ \rmd x.
\end{align*}


\textbf{Estimate of $I_{12}$}: The difficulty in estimating $I_{12}$ lies on the possible derivative loss problem.
 However, this difficulty can be bypassed by using the symmetry structure of the system.

Employing integration by parts, we have
\begin{align}\label{I12_1}
I_{12}&= -\int_{\mathbb{R}^n}
(\partial_t Z^a v\cdot\nabla)  d
 \cdot \partial_t Z^a d \ \rmd x  \nonumber\\
&=-\partial_t\int_{\mathbb{R}^n}
( Z^a v\cdot\nabla)  d
 \cdot \partial_t Z^a d \ \rmd x
+\int_{\mathbb{R}^n}
( Z^a v\cdot\nabla)  \partial_t d
 \cdot \partial_t Z^a d \ \rmd x \nonumber\\
&\quad\ +\int_{\mathbb{R}^n}
( Z^a v\cdot\nabla)  d
 \cdot \partial_{tt} Z^a d \ \rmd x.
\end{align}
By Lemma \ref{lemdecayliq}, the second term on the right hand side of \eqref{I12_1} is controlled by
\begin{align*}
\|Z^a v\|_{L^2} \|\partial_t Z^a d \|_{L^2}
\| \nabla\partial_t d\|_{L^\infty}
\lesssim \langle t\rangle^{-1} (E^{v}_{\kappa} E^{d}_{\kappa+1}E^{d}_{\kappa-1})^{\frac12}.
\end{align*}
For the last term of \eqref{I12_1}, we are going to insert the equation
$\eqref{LiquidHyper_GeDe}_2$ for the orientation field into this expression to show the symmetry:
\begin{align}\label{I12_2}
&\int_{\mathbb{R}^n}
( Z^a v\cdot\nabla)  d
 \cdot \partial_{tt} Z^a d \ \rmd x   \nonumber\\
&=\int_{\mathbb{R}^n}
( Z^a v\cdot\nabla)  d
 \cdot \Delta Z^a d \ \rmd x
 +\int_{\mathbb{R}^n}
( Z^a v\cdot\nabla)  d
 \cdot f^2_{a} \ \rmd x .
\end{align}
 Thanks to Lemma \ref{decaypre}, we get by integration by parts that
\begin{align*}
\int_{\mathbb{R}^n}
( Z^a v\cdot\nabla)  d
 \cdot \Delta Z^a d \ \rmd x~ &\le \| \nabla Z^a d\|_{L^2}
(\| \nabla Z^a v\|_{L^2} \| \nabla d\|_{L^\infty}
+\| Z^a v\|_{L^2} \| \nabla^2 d\|_{L^\infty}) \\[-4mm]\\
&\lesssim \langle t\rangle^{-1}
(E^d_{\kappa+1}E^d_{\kappa-1})^{\frac12}
(\| \nabla Z^a v\|_{L^2}+(E^v_{\kappa})^{\frac12}).
\end{align*}
On the other hand, it holds that

\begin{lem}
\begin{align*}
\int_{\mathbb{R}^n}   (Z^a v\cdot\nabla)  d \cdot f^2_{a } \ \rmd x & \leq
-\frac12\frac{d}{dt} \int_{\mathbb{R}^n} \big|( Z^a v\cdot\nabla)  d\big|^2
\ \rmd x +\langle t\rangle^{-1}(E^v_{\kappa}+E^d_{\kappa-1})^{\frac{1}{2}}E^d_{\kappa+1}\\
&\quad+\|\nabla Z^\kappa v\|^2_{L^2} E^d_{\kappa+1}+\langle t\rangle^{-1}
\|\nabla Z^\kappa v\|_{L^2} E^d_{\kappa+1}+\|\partial_t v\|_{L^\infty}E^d_{\kappa+1}\nonumber.
\end{align*}
\end{lem}
\begin{proof}

Inserting the expression of $f_a^2$ in \eqref{f1f2} into
$\int_{\mathbb{R}^n}( Z^a v\cdot\nabla)  d\cdot f^2_{a } \ \rmd x$, we get
\begin{align}\label{I122}
&\int_{\mathbb{R}^n}
 ( Z^a v\cdot\nabla)  d
 \cdot f^2_{a } \ \rmd x  \nonumber \\
&=-\int_{\mathbb{R}^n}
 ( Z^a v\cdot\nabla)  d
 \cdot \big[(v\cdot\nabla)
(v\cdot\nabla Z^a d+ 2 \partial_t Z^a d )+\partial_t Z^a v \cdot \nabla d
 \big]\ \rmd x
+L,
\end{align}
where
\begin{align*}
L=&\sum_{b+c+e=a} C_a^{b,c}\int_{\mathbb{R}^n}
 ( Z^a v\cdot\nabla)  d \cdot
(\nabla Z^b d\cdot\nabla Z^c d-\partial_t Z^b d\cdot \partial_t Z^c d) Z^e d \ \rmd x \nonumber\\
&-\sum_{b+c+e+f=a} C_a^{b,c,e}\int_{\mathbb{R}^n}
 ( Z^a v\cdot\nabla)  d
 \cdot \big[2 (Z^b v\cdot\nabla) Z^c d\cdot\partial_t Z^e d \big] Z^f d \ \rmd x \nonumber\\
&-\sum_{b+c+e+f+g=a}C_a^{b,c,e,f}\int_{\mathbb{R}^n}
 ( Z^a v\cdot\nabla)  d\cdot\big[  ( Z^b v\cdot \nabla) Z^c d \cdot ((Z^e v\cdot \nabla) Z^f d)  \big] Z^g d \ \rmd x \nonumber\\
&-\sum_{{\substack{b+c+e=a\\e\neq a}}}C_a^{b,c}\int_{\mathbb{R}^n}( Z^a v\cdot\nabla)  d
 \cdot Z^b v\cdot\nabla( Z^c v\cdot \nabla Z^e d) \ \rmd x \nonumber\\
&-\sum_{{\substack{b+c=a\\c\neq a}}}C_a^b\int_{\mathbb{R}^n}( Z^a v\cdot\nabla)  d
 \cdot(2Z^b v\cdot\nabla\partial_t Z^c d+\partial_t Z^c v\cdot\nabla Z^b d) \ \rmd x \nonumber\\
\triangleq&K_1+K_2+K_3+K_4+K_5.
\end{align*}

The first term on the right hand side of \eqref{I122} refers to the highest order term, while the remaining terms denoted by
$L$ refers to the lower order ones.
{The estimate of $L$ is similar but much easier than that of $I_2$,  because $I_2$ contains quadratic terms, while $L$ contains only cubic terms or higher.}

Now let us estimate the right hand side of \eqref{I122} one by one.
The first part of the higher order terms can be bounded by
\begin{align}
&-\int_{\mathbb{R}^n} ( Z^a v\cdot\nabla)  d \cdot \big[(v\cdot\nabla)
(v\cdot\nabla Z^a d+ 2 \partial_t Z^a d) \big]\ \rmd x \nonumber \\
&\lesssim \|\nabla Z^a d \|_{L^2} \|v\|^2_{L^\infty}
( \|\nabla Z^a v \|_{L^2} \|\nabla d\|_{L^\infty}+\| Z^a v \|_{L^2} \|\nabla ^2 d\|_{L^\infty})
\nonumber\\[-4mm]\nonumber\\
&+\|\partial_t Z^a d \|_{L^2} \|v\|_{L^\infty}
\|\nabla Z^a v \|_{L^2} \| \nabla d\|_{L^\infty}+
\|\partial_t Z^a d \|_{L^2} \|v\|_{L^6}
\| Z^a v \|_{L^6} \| \nabla^2 d\|_{L^6} \nonumber\\[-4mm]\nonumber\\
&\lesssim E^d_{\kappa+1}((E^v_{\kappa})^{\frac12}+1)
\| \nabla Z^\kappa v\|^2_{L^2}. \nonumber
\end{align}
For the second part of the higher order terms, now one can see the symmetry is present.
Thus, employing integration by parts, one has
\begin{align}
&-\int_{\mathbb{R}^n}
 ( Z^a v\cdot\nabla)  d
 \cdot (\partial_t Z^a v \cdot \nabla d)
\ \rmd x \nonumber \\
&=-\frac12\frac{d}{dt} \int_{\mathbb{R}^n}
\big|( Z^a v\cdot\nabla)  d\big|^2
\ \rmd x \nonumber
+\int_{\mathbb{R}^n}
 ( Z^a v\cdot \partial_t\nabla)  d
 \cdot ( Z^a v \cdot \nabla d)
\ \rmd x \nonumber \\
&\leq -\frac12\frac{d}{dt} \int_{\mathbb{R}^n}
\big|( Z^a v\cdot\nabla)  d\big|^2
\ \rmd x + \|\nabla Z^\kappa v\|^2_{L^2} E^d_{\kappa+1}
\nonumber.
\end{align}

Now we show the estimate for $L$ in \eqref{I122}.
The estimate is similar to the estimates of $I_1$ to $I_6$ in the higher-order energy estimate for the orientation field. Hence we only sketch the argument.
Due to H\"{o}lder inequality, \eqref{K-S-3D-2}, \eqref{KSWeiIn1}, \eqref{KSWeiIn1}, Lemma \ref{lemmaWeightedNorm} and Lemma \ref{lemdecayliq}, one deduces that
\begin{align*}
K_1&~\lesssim\|Z^a v\|_{L^6} \|\nabla d\|_{L^\infty}
\|\partial Z^{|a|} d\|_{L^2} \|\partial Z^{[|a|/2]} d\|_{L^3} \| Z^{[|a|/2]} d\|_{L^\infty}\\[-4mm]\\
&~\quad+\|Z^a v\|_{L^6} \|\nabla d\|_{L^\infty}
\|\partial Z^{[|a|/2]} d\|_{L^2} \|\partial Z^{[|a|/2]} d\|_{L^6} \| Z^{|a|} d\|_{L^6} \\[-4mm]\\
&~\lesssim \langle t\rangle ^{-1} \|\nabla Z^{\kappa} v\|_{L^2} (E^d_{\kappa+1})^{\frac12}E^d_{\kappa-1} (1+(E^d_{\kappa-1})^{\frac12}),
\end{align*}
and
\begin{align*}
K_2&~\lesssim\|Z^a v\|_{L^6} \|\nabla d\|_{L^\infty} \| Z^{|a|} v\|_{L^6}
\big(\|\partial Z^{|a|} d\|_{L^2}\| \partial Z^{[|a|/2]} d\|_{L^6} \| Z^{[|a|/2]} d\|_{L^\infty} \\[-4mm]\\
&~\quad+\|\partial Z^{[|a|/2]} d\|_{L^2}\| \partial Z^{[|a|/2]} d\|_{L^\infty} \| Z^{|a|} d\|_{L^6}  \big)\\[-4mm]\\
&~\lesssim \langle t\rangle ^{-1} \|\nabla Z^{\kappa} v\|^2_{L^2} (E^d_{\kappa+1})^{\frac12}E^d_{\kappa-1} (1+(E^d_{\kappa-1})^{\frac12}),
\end{align*}
and
\begin{align*}
K_3&~\lesssim\|Z^a v\|_{L^6} \|\nabla d\|_{L^\infty} \| Z^{|a|} v\|_{L^6}
\big( \|\nabla Z^{|a|} d\|_{L^2} \|Z^{[|a|/2]}v \|_{L^\infty} \|\nabla Z^{[|a|/2]} d\|_{L^6} \| Z^{[|a|/2]} d\|_{L^\infty} \\[-4mm]\\
&~\quad+\|\nabla Z^{[|a|/2]} d\|_{L^\infty} \|Z^{|a|}v \|_{L^2} \|\nabla Z^{[|a|/2]} d\|_{L^6} \| Z^{[|a|/2]} d\|_{L^\infty}\\[-4mm]\\
&~\quad+\|\nabla Z^{[|a|/2]} v\|_{L^2}\| \nabla Z^{[|a|/2]} d\|^2_{L^\infty}  \| Z^{|a|} d\|_{L^6}  \big)\\[-4mm]\\
&~\lesssim \langle t\rangle ^{-1} \|\nabla Z^{\kappa} v\|^2_{L^2} (E^d_{\kappa+1} E^v_{\kappa})^{\frac12}E^d_{\kappa-1} (1+(E^d_{\kappa-1})^{\frac12}),
\end{align*}
and
\begin{align*}
K_4\lesssim &~\|Z^a v\|_{L^6} \|\nabla d\|_{L^\infty} \|Z^{|a|}v \|_{L^6} \| Z^{[|a|/2]+1} v\|_{L^6} \|\nabla Z^{[|a|/2]+1} d\|_{L^2} \\[-4mm]\\
&~+\|Z^a v\|_{L^6} \|\nabla d\|_{L^\infty}\|Z^{[|a|/2]}v \|_{L^3} \big(\|\nabla Z^{|a|} v\|_{L^2} \|\nabla Z^{[|a|/2]} d\|_{L^\infty}
\\[-4mm]\\
&~+ \|Z^{[|a|/2]+1} v\|_{L^\infty} \|\nabla Z^{|a|} d\|_{L^2}
+ \|Z^{|a|} v\|_{L^6}  \|\nabla^2 Z^{[|a|/2]} d\|_{L^3}
\big) \\[-4mm]\\
\lesssim &~\langle t\rangle^{-1}\| \nabla Z^{\kappa}\|^2_{L^2} (E^v_{\kappa}E^d_{\kappa-1}E^d_{\kappa+1} )^{\frac12},
\end{align*}
and
\begin{align*}
K_5\lesssim \| Z^a v\|_{L^6} \|\nabla d\|_{L^6}\|Z^{|a|} v\|_{L^6} \|\nabla Z^{|a|} d\|_{L^2}
   \lesssim \|\nabla Z^{\kappa}v\|^2_{L^2} E^d_{\kappa+1}.
\end{align*}
Combining all the above estimates  and noting $E^v_{\kappa}\lesssim\epsilon$ and $E^d_{\kappa-1}\lesssim\epsilon$, we get the lemma.
\end{proof}
\subsubsection{Completing the estimates}
Combining all the above estimates in this subsection, one has
\begin{align*}
&\frac 12 \frac{d}{dt}\int_{\mathbb{R}^n}|\partial Z^a d|^2\ \rmd x-\frac{1}{2}\frac{d}{dt}\int_{\mathbb{R}^n}\big|(v\cdot\nabla) Z^a d\big|^2\ \rmd x\\
&\quad+\frac{d}{dt} \int_{\mathbb{R}^n}
( Z^a v\cdot\nabla)  d \cdot\partial_tZ^a d\ \rmd x
+\frac12\frac{d}{dt} \int_{\mathbb{R}^n}
\big|( Z^a v\cdot\nabla)  d\big|^2
\ \rmd x\nonumber \\
&\lesssim \langle t\rangle^{-1}
(E^v_{\kappa}+E^d_{\kappa-1})^{\frac{1}{2}}
E^d_{\kappa+1}+\|\nabla Z^\kappa v\|^2_{L^2} E^d_{\kappa+1}\\[-4mm]\\
&\quad+\langle t\rangle^{-1}
\|\nabla Z^\kappa v\|_{L^2} E^d_{\kappa+1}
+\|\partial_t v\|_{L^\infty}E^d_{\kappa+1}.
\end{align*}
Summing over $|a|\leq \kappa$, and noting that
\begin{align*}
\sum_{|a|\leq\kappa} \int_{\mathbb{R}^n}
\big|( Z^a v\cdot\nabla)  d\big|^2
\ \rmd x \lesssim E^v_{\kappa} E^d_{\kappa-1}\lesssim \ve E^d_{\kappa-1},
\end{align*}
we can deduce
\begin{align*}
&\sum_{|a|\leq\kappa}\Big(\frac 12 \int_{\mathbb{R}^n}|\partial Z^a d|^2\ \rmd x-\frac{1}{2}\int_{\mathbb{R}^n}\big|(v\cdot\nabla) Z^a d\big|^2\ \rmd x\\
&\quad+ \int_{\mathbb{R}^n}
( Z^a v\cdot\nabla)  d \cdot\partial_tZ^a d\ \rmd x
+\frac12 \int_{\mathbb{R}^n}
\big|( Z^a v\cdot\nabla)  d\big|^2
\ \rmd x \Big)\nonumber \\
&\sim \frac 12 \sum_{|a|\leq\kappa} \int_{\mathbb{R}^n}|\partial Z^a d|^2\ \rmd x
=\frac 12 E^d_{\kappa+1}.
\end{align*}
Here we have used the assumption $E^v_{\kappa}\lesssim\epsilon$ and $E^d_{\kappa-1}\lesssim\epsilon$ again.
This gives \eqref{prioridh}.


\subsection{Lower-order energy estimate for the orientation field}
This subsection is devoted to the lower-order energy estimate for the orientation field $d$.
It turns out that the lower-order energy is uniformly bounded. To this end, we need to obtain the subcritical decay for the nonlinearities or say, $L^1$ integrability in time.

Let $\kappa \geq 9$, $0\leq|a| \leq \kappa-2$.
Taking the $L^2$ inner product of  ${\eqref{LiquidHyper_GeDe}}_2$ with $\partial_t Z^a d$, we have
\begin{eqnarray}\label{HighEnergyd}
\frac 12 \frac{d}{dt}
\int_{\mathbb{R}^n}|\partial Z^a d|^2\ \rmd x
=\int_{\mathbb{R}^n}
f^2_a \cdot  \partial_t Z^a d  \ \rmd x
. \nonumber
\end{eqnarray}
Recalling the expression of $f^2_a$ in \eqref{f1f2}, we will rewrite the right hand side of the above equality as
\begin{align*} 
&-\sum_{b+c=a} C_a^b\int_{\mathbb{R}^3}
(2Z^b v\cdot\nabla\partial_t Z^c d+\partial_t Z^b v\cdot\nabla Z^c d)\cdot  \partial_t Z^a d \ \rmd x\\
&+\int_{\mathbb{R}^3}
\sum_{b+c+e=a} C_a^{b,c}
(\nabla Z^b d\cdot\nabla Z^c d-\partial_t Z^b d\cdot \partial_t Z^c d) Z^e d\cdot \partial_t Z^a d\ \rmd x\\
&-\sum_{b+c+e=a}C_a^{b,c}
\int_{\mathbb{R}^3}Z^b v\cdot\nabla( Z^c v\cdot \nabla Z^e d)\cdot\partial_t Z^a d\ \rmd x\\
&-\sum_{b+c+e+f=a} C_a^{b,c,e} \int_{\mathbb{R}^3}\big[2 (Z^b v\cdot\nabla) Z^c d\cdot\partial_t Z^e d \big] Z^f d
\cdot \partial_t Z^a d \ \rmd x \\
&-\sum_{b+c+e+f+g=a}C_a^{b,c,e,f}\int_{\mathbb{R}^3}
\big[  ( Z^b v\cdot \nabla) Z^c d \cdot ((Z^e v\cdot \nabla) Z^f d)  \big] Z^g d\cdot \partial_tZ^a d \ \rmd x\\
&=M_1+M_2+M_3+M_4+M_5.
\end{align*}
In the sequel, we will estimate $M_1$ to $M_5$ one by one.

We first estimate $M_1$. There are two terms inside the expression for $M_1$. In order to estimate them in a uniform way, we write by integration by parts that
\begin{align*}
&\sum_{b+c=a} C_a^b\int_{\mathbb{R}^3}
(\partial_t Z^b v\cdot\nabla Z^c d)\cdot  \partial_t Z^a d \ \rmd x
=\sum_{b+c=a} C_a^b\partial_t\int_{\mathbb{R}^3}
( Z^b v\cdot\nabla Z^c d)\cdot  \partial_t Z^a d \ \rmd x \\
&-\sum_{b+c=a} C_a^b\int_{\mathbb{R}^3}
( Z^b v\cdot\nabla\partial_t Z^c d)\cdot  \partial_t Z^a d \ \rmd x
-\sum_{b+c=a} C_a^b\int_{\mathbb{R}^3}
( Z^b v\cdot\nabla Z^c d)\cdot  \partial^2_t Z^a d \ \rmd x.
\end{align*}
Thus $M_1$ can be bounded by
\begin{align*}
-\sum_{b+c=a} C_a^b\partial_t\int_{\mathbb{R}^3}
( Z^b v\cdot\nabla Z^c d)\cdot  \partial_t Z^a d \ \rmd x
+\int_{\mathbb{R}^3}
 |Z^{|a|} v| |\partial Z^{|a|} d| |\partial^2 Z^{|a|} d| \ \rmd x.
\end{align*}
Now we estimate the last term in the above expression.
For the integral domain of $\{ r\leq \langle t \rangle /2 \}$, one deduces from Lemma \ref{lemmaWeightedNorm} and the Sololev inequality $\|u\|_{L^\infty}\lesssim \|\nabla u\|^{\frac12}_{L^2}\|\nabla^2 u\|^{\frac12}_{L^2}$ that
\begin{align*}
&\int_{r\leq \langle t \rangle /2}
 |Z^{|a|} v| |\partial Z^{|a|} d| |\partial^2 Z^{|a|} d| \ \rmd x \\
&\lesssim
\langle t\rangle^{-1} \|Z^{|a|} v\|_{L^\infty} \|\partial Z^{|a|} d\|_{L^2}
\|\langle t-r\rangle\partial^2 Z^{|a|} d\|_{L^2(r\leq \langle t \rangle /2)} \\[-4mm]\\
&\lesssim \langle t\rangle^{-1} \|\nabla Z^{\kappa-1} v\|^{\frac{1}{2}}_{L^2} (E^v_{\kappa})^{\frac14}(E^d_{\kappa-1}E^d_{\kappa+1})^{\frac12}.
\end{align*}
For the integral region of $\{ r\geq \langle t \rangle /2\}$, we have from \eqref{K-S-3D-2} that
\begin{align*}
&\int_{r\geq \langle t \rangle /2}
 |Z^{|a|} v| |\partial Z^{|a|} d| |\partial^2 Z^{|a|} d| \ \rmd x \\
&\lesssim
\langle t\rangle^{-1} \|r Z^{|a|} v\|_{L^\infty(r\geq \langle t \rangle /2)} \|\partial Z^{|a|} d\|_{L^2}
\|\partial^2 Z^{|a|} d\|_{L^2} \\[-4mm]\\
&\lesssim \langle t\rangle^{-1} \|\nabla Z^{\kappa} v\|^{\frac{1}{2}}_{L^2} (E^v_{\kappa})^{\frac14}(E^d_{\kappa-1}E^d_{\kappa+1})^{\frac12}.
\end{align*}

Next, we write $M_2$ as:
\begin{eqnarray*}
M_2=\int_{\mathbb{R}^3}
\sum_{b+c+e=a} C_a^{b,c}
(\nabla Z^b d\cdot\nabla Z^c d-\partial_t Z^b d\cdot \partial_t Z^c d) Z^e d\cdot \partial_t Z^a d\ \rmd x.
\end{eqnarray*}
In the integral domain of $\{r\leq \langle t\rangle/2\}$,
employing \eqref{KSWeiIn2}, \eqref{KSWeiIn3} and Lemma \ref{lemmaWeightedNorm} yields
\begin{align*}
&\int_{r\leq \langle t\rangle/2}
\sum_{\tiny\substack{ b+c+e=a\\|e|\geq {|a|}/2 }}C_a^{b,c}
(\nabla Z^b d\cdot\nabla Z^c d-\partial_t Z^b d\cdot \partial_t Z^c d) Z^e d\cdot \partial_t Z^a d\ \rmd x\\
&\lesssim
\sum_{|e|\geq {|a|}/2}
\|\partial Z^{[|a|/2]} d\|_{L^6(r\leq \langle t\rangle/2)}
\|\partial Z^{[|a|/2]} d\|_{L^6(r\leq \langle t\rangle/2)}
\|Z^e d\|_{L^6}\|\partial_t Z^a d\|_{L^2}\\
&+\sum_{|e|\leq \frac{|a|}{2}}\|\partial Z^{|a|} d\|_{L^6(r\leq \langle t\rangle/2)}
\|\partial Z^{[|a|/2]} d\|_{L^3(r\leq \langle t\rangle/2)} \|Z^e d\|_{L^\infty} \| \partial_t Z^a d\|_{L^2}\\
&\lesssim\langle t\rangle^{-\frac32} E^d_{\kappa-1} (E^d_{\kappa+1})^{\frac12}(1+(E^d_{\kappa-1})^{\frac12}).
\end{align*}
To estimate the integral domain of $\{r\geq \langle t\rangle/2\}$, we need to use the null condition. To this end, we first write
\begin{align*}
&\sum_{b+c+e=a} C_a^{b,c}\int_{r\geq\langle t\rangle/2}
(\nabla Z^b d\cdot\nabla Z^c d-\partial_t Z^b d\cdot \partial_t Z^c d) Z^e d\cdot \partial_t Z^a d\ \rmd x \\
&=\sum_{b+c+e=a} C_a^{b,c}\int_{r\geq\langle t\rangle/2}
(\omega_i\partial_t+\nabla_i) Z^b d \cdot(\omega_i\partial_t-\nabla_i) Z^c d Z^e d\cdot \partial_t Z^a d\ \rmd x \\
&\lesssim\sum_{b+c+e=a} \int_{r\geq\langle t\rangle/2}
|(\omega_i\partial_t+\nabla_i) Z^b d| |\partial Z^c d| |Z^e d| |\partial_t Z^a d|\ \rmd x,
\end{align*}
where $\omega=x/r$.
Consequently, employing \eqref{K-S-3D-1} and Lemma \ref{lemWeiGood},
the above can be further controlled by
\begin{align*}
&\sum_{|e|\leq |a|/2}\|(\omega_i\partial_t+\nabla_i) Z^{|a|} d\|_{L^\infty(r\geq \langle t\rangle/2)}
\|\partial Z^{|a|} d\|_{L^2} \|Z^e d\|_{L^\infty} \|\partial_t Z^a d\|_{L^2}\\
&+\sum_{|e|\geq {|a|}/2}\|(\omega_i\partial_t+\nabla_i)
Z^{[|a|/2]} d\|_{L^\infty(r\geq \langle t\rangle/2)}
\|\partial Z^{[|a|/2]} d\|_{L^3} \|Z^e d\|_{L^6} \|\partial_t Z^a d\|_{L^2}\\
&\lesssim \langle t\rangle^{-\frac32} (E^d_{\kappa+1})^{\frac12} E^d_{\kappa-1}(1+E^d_{\kappa-1})^{\frac12}.
\end{align*}
Here we have used the spatial decomposition along radial and reverse direction:
\begin{equation*}
\nabla=\frac{x}{r}\partial_r-\frac{\omega}{r}\wedge\Omega.
\end{equation*}

Now we estimate $M_3$, $M_4$ and $M_5$ which
contain  cubic, quartic and quintic terms but no quadratic term.
Thanks to \eqref{K-S-3D-1}, \eqref{KSWeiIn1} and Lemma \ref{lemmaWeightedNorm}, one gets
\begin{align*}
M_3&=-\sum_{b+c+e=a}C_a^{b,c}
\int_{\mathbb{R}^3}Z^b v\cdot\nabla( Z^c v\cdot \nabla Z^e d)\cdot\partial_t Z^a d\ \rmd x\\
&\lesssim
(\int_{r\geq \langle t\rangle/2}+\int_{r\geq \langle t\rangle/2})\ |Z^{|a|} v| |Z^{|a|+1} v| |\nabla Z^{|a|+1} d| |\partial_t Z^a d|\ \rmd x\\
&\leq  \|Z^{|a|} v\|_{L^\infty(r\geq \langle t\rangle/2)} \|Z^{|a|+1} v\|_{L^\infty} \|\nabla Z^{|a|+1} d\|_{L^2} \|\partial_t Z^a d\|_{L^2}\\[-4mm]\\
&\quad+ \|Z^{|a|+1} v\|^2_{L^6} \|\nabla Z^{|a|+1} d\|_{L^6(r\leq \langle t\rangle/2)} \|\partial_t Z^a d\|_{L^2}\\[-4mm]\\
&\leq \langle t\rangle^{-1}\| \nabla Z^\kappa v\|^{\frac{1}{2}}_{L^2} (E^v_{\kappa})^{\frac32} (E^d_{\kappa+1}E^d_{\kappa-1})^{\frac12}.
\end{align*}
Moreover, we have
\begin{align*}
M_4&=-\sum_{b+c+e+f=a} C_a^{b,c,e} \int_{\mathbb{R}^3}\big[2 (Z^b v\cdot\nabla) Z^c d\cdot\partial_t Z^e d \big] Z^f d
\cdot \partial_t Z^a d \ \rmd x \\
&\lesssim \sum_{b+c+e+f=a}\int_{\mathbb{R}^3} |Z^b v| |\partial Z^c d| |\partial Z^e d| |Z^f d|
|\partial_t Z^a d| \ \rmd x.
\end{align*}
By Sobolev inequalities and Lemma \ref{lemdecayliq}, the above can be further bounded by
\begin{align*}
&\sum_{\frac{|a|}{2}\leq |f|\leq |a|}
\|Z^{|a|} v\|_{L^6} \|\partial Z^{|a|} d\|_{L^6} \|\partial Z^{[|a|/2]} d\|_{L^\infty} \|Z^f d\|_{L^6}
\|\partial_t Z^a d\|_{L^2} \\
&+\sum_{|f|\leq\frac{|a|}{2}}\|Z^{|a|} v\|_{L^\infty} \|\partial Z^{|a|} d\|_{L^2} \|\partial Z^{[|a|/2]} d\|_{L^\infty} \|Z^f d\|_{L^\infty}
\|\partial_t Z^a d\|_{L^2}\\
&\lesssim \langle t\rangle^{-1}
\|\nabla Z^{\kappa} v\|^{\frac{1}{2}}_{L^2}\|\nabla Z^{\kappa-1} v\|^{\frac{1}{2}}_{L^2}
 (E^d_{\kappa-1})^{\frac32} \big[(E^d_{\kappa+1})^{\frac12}+1\big].
\end{align*}
The term $M_5$ can be estimated as
\begin{align*}
M_5&=-\sum_{b+c+e+f+g=a}C_a^{b,c,e,f}\int_{\mathbb{R}^3}
\big[  ( Z^b v\cdot \nabla) Z^c d \cdot ((Z^e v\cdot \nabla) Z^f d)  \big] Z^g d\cdot \partial_tZ^a d \ \rmd x\\
&\lesssim
\sum_{|e|+|f|+|g|\leq |a|}\int_{\mathbb{R}^3}
|Z^{|a|} v|^2 |\nabla Z^e d|
|\nabla Z^f d|  |Z^g d| |\partial_tZ^a d| \ \rmd x \\
&\lesssim \sum_{|g|\leq \frac{|a|}{2}}\|Z^{|a|} v\|^2_{L^6} \|\partial Z^{|a|} d\|_{L^2} \|\partial Z^{[|a|/2]} d\|_{L^6} \| Z^g d\|_{L^\infty}
 \| \partial_tZ^a d\|_{L^2}\\
&\quad+\sum_{\frac{|a|}{2}\leq|g|\leq |a|} \|Z^{|a|} v\|^2_{L^6} \|\partial Z^{|a|} d\|_{L^2} \|\partial Z^{[|a|/2]} d\|_{L^\infty} \|Z^g d\|_{L^6}
 \| \partial_tZ^a d\|_{L^2}.  \\
 &\lesssim \| \nabla Z^{\kappa} v\|^2_{L^2} (E^d_{\kappa-1})^{\frac32}(1+E^d_{\kappa-1})^{\frac12}.
\end{align*}

Combining all the estimates of $M_1$ to $M_5$ leads
\begin{align*}
&\frac{d}{dt}\int_{\mathbb{R}^n}\big[\frac 12 |\partial Z^a d|^2+
\sum_{b+c=a} C_a^b
( Z^b v\cdot\nabla Z^c d)\cdot  \partial_t Z^a d\big] \ \rmd x \\
&\lesssim \langle t\rangle^{-1} \|\nabla Z^{\kappa} v\|^{\frac{1}{2}}_{L^2}(E^v_{\kappa})^{\frac14}(E^d_{\kappa-1}E^d_{\kappa+1})^{\frac12}
+\langle t\rangle^{-\frac32} E^d_{\kappa-1} (E^d_{\kappa+1})^{\frac12}
+\| \nabla Z^{\kappa} v\|^2_{L^2}E^d_{\kappa-1}.
\end{align*}
Summing over $|a|\leq \kappa-2$, and noting that
\begin{align*}
&\sum_{|a|\leq\kappa-2}\int_{\mathbb{R}^n}\big[\frac 12 |\partial Z^a d|^2+
\sum_{b+c=a} C_a^b
( Z^b v\cdot\nabla Z^c d)\cdot  \partial_t Z^a d\big] \ \rmd x \\
&\sim \sum_{|a|\leq\kappa-2}\frac 12\int_{\mathbb{R}^n} |\partial Z^a d|^2
 \ \rmd x=\frac12 E^d_{\kappa-1}.
\end{align*}
Thus we finish the proof of \eqref{prioridl}.

\section*{Acknowledgement.}

The authors would like to thank Professor Fanghua Lin for suggesting this problem and his helpful discussion.
We are also grateful to the hospitality of Courant Institute of Mathematical Sciences  where this work was carried out.
Y. Cai was sponsored by the China Scholarship Council (No. 201606100111) for one year at New York University, Courant Institute of Mathematics Sciences.
W. Wang is supported by NSF of China under Grant No. 11922118, 11931010 and 11771388, and the Young Elite Scientists
Sponsorship Program by CAST.

\end{document}